\documentclass[a4paper]{article}
\usepackage{
multirow,authblk,
amssymb,amsmath,
amsthm,
bbm,
bm,
booktabs,
}
\usepackage[left=2cm,
right=2cm,top=3cm,
bottom=3cm]{geometry}
\usepackage[utf8]{inputenc}
\usepackage[ruled,procnumbered]{algorithm2e}
\usepackage{hyperref}
\DeclareUnicodeCharacter{00A0}{ }
\newcommand{\overbar}[1]{\mkern 0.5mu\overline{\mkern-0.5mu#1\mkern-2.0mu}\mkern 2.0mu}
\newcommand{\be}{\begin{equation}}
 \newcommand{\func}[1]{\mbox{\rm #1}\,}
\newcommand{\la}{\left\langle  }
\newcommand{\ra}{\right\rangle  }
\providecommand{\{}{\left\{}
\providecommand{\}}{\right\}}
\providecommand{\norm}[1]{\lVert#1\rVert_{\infty}}
\providecommand{\normt}[1]{\lVert#1\rVert}
\providecommand{\normo}[1]{\lVert#1\rVert_{1}}
\providecommand{\law}{\leftarrow}
\newcommand{\ee}{\end{equation}}
\newtheorem{asump}{Assumption}{\bf}{\it}
\theoremstyle{plain}
\newtheorem{theorem}{Theorem}
\newtheorem{lemma}{Lemma}
\newtheorem{proposition}{Proposition}
\theoremstyle{remark}
\newtheorem*{remark}{Remark}
\theoremstyle{definition}
\newtheorem{definition}{Definition}[section]
\newtheorem{corollary}{Corollary}
\begin{document}
\title{A version of bundle method with linear programming}
\author{Shuai Liu\thanks{Email: \texttt{liushuai04235@gmail.com} Research supported by the Australian Research Council under Discovery Grant DP12100567.}}
\author{Andrew Eberhard\thanks{Email: \texttt{andy.eb@rmit.edu.au} Research supported by the Australian Research Council under Discovery Grant DP12100567.}}
\author{Yousong Luo\thanks{Email: \texttt{yluo@rmit.edu.au}}}
\affil{School of Mathematical and Geospatial Sciences, \\RMIT University, Melbourne, VIC 3001, Australia}
\date{July 7, 2015}
\maketitle
\begin{abstract}
Bundle methods have been intensively studied for solving both convex and nonconvex optimization problems. In most of the bundle methods developed thus far, at least one quadratic programming (QP) subproblem needs to be solved in each iteration. In this paper, we exploit the feasibility of developing a bundle algorithm that only solves linear subproblems. 
We start from minimization of a convex function and show that the sequence of major iterations converge to a minimizer. For nonconvex functions we  consider functions that are locally Lipschitz continuous and prox-regular on a bounded level set, and minimize the cutting-plane model over a trust region with infinity norm. The para-convexity of such functions allows us to use the locally convexified model and its convexity properties. Under some conditions and assumptions, we study the convergence of the proposed algorithm through the outer semicontinuity of the  proximal mapping. Encouraging results of preliminary numerical experiments on standard test sets are provided.
\end{abstract}

\maketitle
\section{Introduction}
As a generalization of nonlinear programming (NLP), nonsmooth optimization (NSO) has broader application areas and more theoretical challenges. Traced back to as early as 1962 \cite{Shor19621}, NSO has been studied intensively hitherto with more and more new methods being developed. 
 Generally speaking, methods in unconstrained NSO fall into the following categories: subgradient methods \cite{Shor1985,kiwiel1983aggregate,Bagirov2013}, gradient sampling methods \cite{burke2005robust,Tang2014}, modified Newton \cite{Qi1993} or quasi-Newton \cite{Lewis2013a} methods, proximal point methods \cite{rockafellar1976monotone}, derivative free methods \cite{vaz2007particle}, and bundle methods \cite{Lemarechal1975,Wolfe1975,kiwiel1990proximity,Schramm1992}. NSO has a much broader scope than NLP and theory can be developed for very generic 
 functions classes ie. Lipschitz continuous etc.  So in order to develop more efficient algorithms 
and stronger theoretic results, various special structures of the objective function
are often assumed. Special algorithms have developed, tailored for well-structured problems such as convex composite \cite{jeyakumar1995convex}, partially separable \cite{lukvsan2006variable}, etc.  
  Bundle method is probably the most intensively researched methods for nonsmooth optimization. It was created independently by Claude Lemar\'{e}chal \cite{Lemarechal1975} and Philip Wolfe \cite{Wolfe1975} in 1975. Since then a great number of variants of bundle methods have been developed, such as proximal bundle \cite{kiwiel1990proximity,feltenmark2000dual}, trust region bundle \cite{Schramm1992,dempe2001bundle}, splitting bundle \cite{Fuduli2013}, and redistributed bundle \cite{Hare2010}.
Bundle methods grew out of cutting plane methods which often showed  
  a great deal of instability (see page 276 of \cite{Hiriart-Urruty1993a}). To correct this a better approximation of the whole subdifferential is made and stabilization of the descent step is also incorporated. 
  
Consider any optimization problems of the form 
\be
\label{p} \underset{x\in\mathbbm{R}^n}{\operatorname{minimize}}\quad f(x)
\ee
where the objective function $f:\mathbbm{R}^n\rightarrow \mathbbm{R}$ is locally Lipschitz continuous. The traditional bundle method solves a parametric quadratic subproblem in each iteration to obtain a search direction with possible options to follow with a line search. The advantage of employing a quadratic subproblem is that the optimal solution is unique due to the strict convexity, furthermore, the solution can be explicitly expressed in terms of the current iteration point and an average of some subgradients. 
However this convenience must be weighed up against the inconvenience  of having 
to solve a QP at each iteration which can be time-consuming, especially for large scale problems.  One of the motivations of this paper is to overcome the necessity of solving a QP
 in each iteration. Hence, we propose a new version of bundle algorithm for solving \eqref{p} with linear subproblems only. Surprisingly, one finds that even without the explicit expression of the
  solution, and even without line search, our algorithm can still converge to stationary points for nonconvex problems. 
  
  Our approach uses two key tools to achieve this end. First, is a subdifferential
 approximation. Traditional bundle methods use a convex combination of subgradients
(an average) to approximate a selection from an approximation of the whole
subdifferential set (the subgradient selection technique). Traditional subgradient
methods, including gradient sampling methods, use minimum norm subgradient as a
 replacement for the gradient, a calculation involving the solution of a QP. In our method we use a trust
region method based on a linear cutting plane model 
(of a related local convexification of $f$) and in the theoretical analysis we use
 $\frac{f(x)-\bar{f}}{\norm{x-P(x)}}$, where $\bar{f} $ is the minimum value of $f$ and $P(x)$ is the projection of $x$ onto the optimal solution set, as a lower bound on the norm of the minimum norm element of the subdifferential. This alternative for the  subdifferential approximation plays a significant role in our convergence analysis. Secondly, we use a local convexification technique developed in \cite{Hare2010}. For nonconvex functions that are prox-regular and Lipschitz continuous, we show that there exists a number `$a$' such that $f(\cdot)+\frac{a}{2}\|\cdot-\bar{x}\|^2$ is a restriction, to a level set of $f$, of a globally convex function.  Clearly when $f$ is convex 
it suffices to take $a=0$. Unlike \cite{Hare2010}, where the convexification parameter `$a$' is eventually stabilized, we only need `$a$' to be bounded. This allows us to exploit the outer semicontinuity of the subdifferential and that of the associated \emph{proximal mapping}, while traditional bundle methods use the outer semicontinuity of the $\epsilon$-subdifferential.  
   
The use of a linear model in a trust region method is a relatively new idea and our work follows that of \cite{Linderoth2003} where a related approach is used to solve a 
large scale optimization problem that arise out of stochastic programming. 
Standard theory for trust region methods use quadratic model functions. This is usually 
justified via the imposition of sufficient differentiability 
assumption on the objective function. Such logic is 
less compelling in the context of nonsmooth optimization. Furthermore, one usually
associates a bundle--trust--region method  to an approach 
that absorbs the trust region into a quadratic
penalty to control the step length. We instead directly  impose an infinity box norm that is 
handled in more or less a tradition manner for a trust region method. Unlike 
\cite{Linderoth2003} we are able to handle a generic class of prox-regular, 
locally Lipschitz function, greatly extending the applicability of this approach. 

This paper is organized as follows. Section 2 contains the properties of our objective function as preliminary knowledge and Section 3 includes a description of a version of bundle method with linear subproblem for convex optimization. In Section 4 we derive the method for nonconvex optimization and we analyse the convergence of the algorithm in Section 5. Preliminary numerical tests are presented in Section 6.  
   
In this paper, $\normt{\cdot}$, $\normo{\cdot}$ and $\norm{\cdot}$ denote the two norm, one norm and infinity norm, respectively. Denote by  
$ \text{lev}_{ b}f$  the lower level set of $f$ defined by $\{x\in\mathbbm{R}^n \big| f(x)\leq b\}$ and $B(x,\epsilon)$ the closed ball centered at $x$ with radius $\epsilon$.  The convex hull of a set $C\in \mathbbm{R}^n$ is denoted by co $C$, the domain of function $f$ is $\text{dom }f$ and  the interior of a set $C$ is $\text{int }C$.
\section{Properties of the objective function}
For the reader's convenience we collect in this section some standard definitions and properties we will be utilizing in our development. 
\begin{definition}[subdifferential]
Let $f:\mathbbm{R}^n \rightarrow\overbar{\mathbbm{R}}$ be finite at $\bar{x}$.
\begin{enumerate}
\item The set 
\be
\hat{\partial}f(\bar{x})\mathrel{\mathop:}=\left\{s\in\mathbbm{R}^n\big |\ \liminf\limits_{x\rightarrow\bar{x}\atop x\not=\bar{x}}\frac{f(x)-f(\bar{x})- \la  s,x-\bar{x}\ra }{||x-\bar{x}||}\geq 0\right\}\notag
\ee
is called the \emph{Fr\'echet subdifferential} or \emph{regular subdifferential }of $f$ at $\bar{x}$ with elements called \emph{Fr\'echet subgradients} or \emph{regular subgradients }of $f$ at $\bar{x}$.  When $|f(\bar{x})|=\infty$ then $\hat{\partial}f(\bar{x})\mathrel{\mathop:}=\emptyset$.
\item The set $\partial f(\bar{x})\mathrel{\mathop:}=\limsup\limits_{x\rightarrow\bar{x}\atop f(x)\rightarrow f(\bar{x})}\hat{\partial}f(\bar{x})  $ is called the \emph{basic subdifferential }of $f$ at $\bar{x}$ with elements called \emph{basic subgradients} of $f$ at $\bar{x}$. When  
$|f(\bar{x})|=\infty$ then ${\partial}f(\bar{x})\mathrel{\mathop:}=\emptyset$.
\end{enumerate}
\end{definition}

\begin{definition}[prox-regularity,  Definition 13.27,  \cite{Rockafellar1998}]
A function $f:\mathbbm{R}^n \rightarrow\overbar{\mathbbm{R}}$ is \emph{prox-regular} at $ \bar{x} $ for 
$ \bar{v}$ with respect to $\epsilon$ and $a$ if $f$ is finite and locally lower semicontinuous (l.s.c) at $\bar{x}$ with $\bar{v}\in\partial f(\bar{x})$, and there exist $\epsilon > 0$ 
and $a \geq 0$ such that
\be\label{prox-regularInequality}
f(x') \geq f(x) + \la  v,x'-x\ra-\frac{a}{2}
||x' - x||^2 \ \forall\ x'\in B(\bar{x},\epsilon)
\ee
when $||x-\bar{x}||<\epsilon$, $v \in \partial f(x),\ ||v-\bar{v}||<\epsilon,\ f(x) < f(\bar{x}) + \epsilon.$
When this holds for all $\bar{v} \in \partial f(\bar{x})$, $f$ is said to be prox-regular at $\bar{x}$.
\end{definition}
\begin{definition}[para-convexity]\label{paraconvexity}
Given a point $\bar{x}\in\mathbbm{R}^n$ and a real number $\epsilon>0$, a function $f:\mathbbm{R}^n \rightarrow\overbar{\mathbbm{R}}$ is \emph{para-convex} on $B(\bar{x},\epsilon)$ with respect to $a$ if there exists $a\geq 0$ such that the function $f(\cdot)+\frac{a}{2}||\cdot||^2$ is convex on $B(\bar{x},\epsilon)$. 
\end{definition}
\begin{definition}[proximal mapping]\label{proximal}
For a proper, l.s.c. function $f : \mathbbm{R}^n \mapsto \bar{\mathbb{R}}$ and a parameter $a > 0$, the \emph{Moreau envelope function}
$e_ a f$ and \emph{proximal mapping} (or \emph{proximal point mapping}) $P_ a f$ are defined by
\begin{eqnarray}
e_a f(x) := \inf_{w\in\text{dom} f}\{f(w) +\frac{a}{2}\normt{w - x}^2\},\\
\label{pro:map} P_a f(x) := \arg\min_{w\in\text{dom} f}\{f(w) +\frac{a}{2}\normt{w - x}^2\}.
\end{eqnarray}
\end{definition}
If the proper l.s.c function $f$ is bounded from below then $P_{a}f(x)$ is nonempty and compact for all $(x,a)\in\mathbb{R}^n\times \mathbb{R}_{>0}$, and the mapping $\mathbb{R}^n\times \mathbb{R}_{>0}\ni(x,a)\mapsto e_{a}f(x)$ is continuous.
From Definition \ref{proximal} we see that if $p\in P_{a}f(x)$ then $f(p)\leq e_{a}f(x)$ and $a(p-x)\in\partial f(p)$. We also have $e_{a}f(x)\leq f(x)$ for all $x\in\mathbbm{R}^n$ and $e_{a}f(x)= f(x)$ if and only if $x\in P_{a}f(x)$.
\begin{definition}[outer semicontinuity]
A set valued mapping $S:\mathbbm{R}^n\rightrightarrows\mathbbm{R}^m$ is \emph{outer semicontinuous} at $\bar{x}$ if 
\[
\{u\ |\ \exists\ x^j\rightarrow \bar{x},\ \exists\ u^j\rightarrow u\text{ with }u^j\in S(x^j)\}= S(\bar{x}).
\]
\end{definition}
We note that both subdifferential and proximal mapping are outer semicontinuous. Additionally, the mapping $\mathbb{R}^n\times \mathbb{R}_{>0}\ni(x,a)\mapsto P_{a}f(x)$ is outer semicontinuous.

The following proposition is from lemma 2.2 of \cite{Eberhard2001}.
\begin{proposition}\label{prox-para}
If a function $f:\mathbbm{R}^n \rightarrow\overbar{\mathbbm{R}}$ is locally Lipschitz continuous and prox-regular at $\bar{x}$ then there exist $\epsilon$ and `$a$' such that $f$ is para-convex on $B(\bar{x},\epsilon)$ with respect to `$a$'. 
\end{proposition}
\begin{remark}\label{con:prop}
If  $f:\mathbbm{R}^n \rightarrow\overbar{\mathbbm{R}}$ is a convex function then it is useful to note that  
\begin{equation}
f(x)=\sup \left\{f(y)+s^T(x-y)|y\in \mathbb{R}^n,\ s\in \partial f(y)\right\}.
\end{equation}
\end{remark}
\section{LP-bundle method : The convex case}
Consider minimizing a convex function $f$ on $\mathbbm{R}^n$. Denote $S$ the set of minimizers of $f$. Then $S$ is closed and convex. Assuming $S$ is nonempty, the projection operator $P(\cdot)$ onto $S$ is well defined. In reference \cite{Linderoth2003}, a bundle trust-region method was proposed to solve a two-stage stochastic linear programming problem. We show that this method can be generalized to minimize any convex and locally Lipschitz continuous functions. We refer to the generalized method as bundle method with linear programming for convex optimization (LPBC).
Given an auxiliary point $y_i$ and a subgradient $s_i\in\partial f(y_i)$, 
a cutting-plane function is a linear mapping
\be\label{cuttinfplanefunction1}
x\mapsto f(y_i)+\la  s_i,x-y_i\ra.
\ee
The cutting-plane model of $f(x)$ is constructed by the point-wise maximum of the cutting-plane functions as follows:  
\begin{equation}\label{mod:con}
m(x)=\max\limits_{i\in I}\{f(y_i)+\la  s_i,x-y_i\ra\},
\end{equation}
where $I$ is the index set of auxiliary points. LPBC applies the model function on a trust region generated by infinity norm so that it solves the following subproblem sequentially, 
\begin{equation}
\begin{split}\label{transit}
\min\limits_{x\in \mathbbm{R}^n}&\quad m(x)\\
\mathrm{subject\ to}&\quad||x-\bar{x}||_\infty\leq \Delta,
\end{split}
\end{equation}
where $\bar{x}$ is the current best candidate for a minimizer of $f$ and $\Delta$ is the trust region radius. Adopting a scalar variable, problem \eqref{transit} is equivalent to the following linear programming problem 
\begin{subequations}\label{subpz2}
\begin{align} 
\min\limits_{(x,z)\in \mathbbm{R}^{n+1}}&\qquad \qquad\qquad z\\
\mathrm{subject\ to}&\quad f(y_i)+\la  s_i,x-y_i\ra\leq z,\ \forall\ i\in I,\label{subp:2}\\
\  & \qquad\quad||x-\bar{x}||_\infty\leq \Delta.
\end{align}
\end{subequations}

During the $k$th iteration, LPBC solves several linear problems with different model functions `$m$' and possibly different trust region radii $\Delta$ before a new iterate $x^{k+1}$ is identified. Hence LPBC refers to $x^k$ and $x^{k+1}$ as \textit{major iterates} and $x^{kl},\ l=0,1,2,\cdots$ obtained by solving the current linear problem as \textit{minor iterates}. We will also use $
x^*$ to denote minor iterates when it is not necessary to identify the iteration indices.
The subscript $kl$ and sometimes $(k,l)$ means $l$ minor iterations have been executed after $k$-th major iteration. Consequently, the $\bar{x}$, $\Delta$ and $I$ in \eqref{subpz2} are replaced by $x^k$, $\Delta^k_l$ and $I(k,l)$.  
After solving subproblem \eqref{subpz2}, an optimal solution $(x^{kl},z^{kl})$ is obtained. And $x^{kl}$ will be accepted as new iterate $x^{k+1}$ if it yields substantial reduction in the real objective $f$, otherwise the model function will be refined by adding and deleting cutting planes. The substantial reduction in the value of $f$ is measured by its quotient with the reduction of model value, i.e. $m^k_l(x^k)-m^k_l(x^{kl})$. 
LPBC updates the model $m$ in a way such that the following conditions hold:
\begin{equation}\label{modCond1}
m^k_l(x^k)=f(x^k),\quad \text{for all index }k,\ l.
\end{equation}
\begin{equation}\label{modCond2}
m^k_l \textrm{ is a convex, piecewise linear lower underestimate of }f ,\ l=1,2,\cdots.
\end{equation}
Specifically, to obtain $m^k_{l+1}$, LPBC flushes all the cutting planes except  the following two types.
\begin{itemize}
\item The cutting plane is generated at $x^k$. Thus the cutting plane $f(x^k)+{s^k}^T(x-x^k)$ is always kept in the linear subproblem during the $k$th major iteration;\label{modelup1}
\item The cutting plane is active at $x^{kl}$ with positive Lagrange multiplier. \label{modelup3}
\end{itemize}
LPBC adds the new cutting plane generated at 
$x^{kl}$, $f(x^{kl})+{s^k_l}^T(x-x^{kl})$ to the model $m^k_{l+1}$.
\SetAlgoShortEnd
\begin{procedure}
\SetAlgoLined
\caption{LPBC Updating Trust Region()}\label{btrutr}
Define 
\begin{equation}\label{rho:c}
\rho_l^k\mathrel{\mathop:}=\frac{f(x^k)-f(x^{kl})}{f(x^k)-m^k_l(x^{kl})}
\end{equation}
\uIf{$\rho^k_l>\eta_3$ \textbf{ and }$\norm{x^{kl}-x^k}>0.9\Delta$}{$\Delta\law\min\{\alpha_2\Delta,\Delta_{\max}\}$}
\ElseIf{$ \rho^k_l<-\frac{1}{\min\{1,\Delta\}}$}{
 $\Delta\law\alpha_1\Delta$\label{different}\;
 }
\end{procedure}

In the definition of $\rho^k_l$ in \eqref{rho:c}, the denominator $f(x^k)-m^k_l(x^{kl})$ is the reduction of model $m(x) $ from $x^k$ to $x^{kl}$ due to condition \eqref{modCond1}. 
We will use the notion \emph{linearization error}, the difference between the value of a function and the value of a cutting-plane function. Consider a cutting plane of a generic function, as defined in \eqref{cuttinfplanefunction1}. The linearization error of this cutting plane at $\bar{x}$ is 
\begin{equation}\label{eiTidle}
\tilde{e}_i\mathrel{\mathop:}=f(\bar{x})-[f(y_i)+\la s_i,\bar{x}-y_i\ra ].
\end{equation}
In the LPBC algorithm, the linear subproblem \eqref{subpz2} was used. Applying the KKT condition to \eqref{subpz2}, we can deduce the explicit expression of the model reduction of LPBC as following, 
 \be\label{con:modRed}
 m(\bar{x})-m(x^*) =\begin{cases}
 \sum\limits_{i\in \bar{I}}\lambda_i\tilde{e}_i,&\text{if } \norm{x^*-\bar{x}}<\Delta;\\
\sum\limits_{i\in \bar{I}}\lambda_i\tilde{e}_i+\Delta
\normo{\sum\limits_{i\in \bar{I}}\lambda_i s_i}, &\text{if } \norm{x^*-\bar{x}}=\Delta.
 \end{cases}
 \ee 
where $\bar{I}$ is the index set for active constraints in \eqref{subp:2} associated with an optimal solution $(x^*,z^*)$, and $\lambda_i$ for $i\in \bar{I}$ are the corresponding Lagrangian multipliers associated with $(x^*,z^*)$. Since $f$ is convex, all the linearization errors $\tilde{e}_i$ will be nonnegative. We also have  
\begin{equation}\label{con:epsSubd}
\sum\limits_{i\in \bar{I}}\lambda_i s_i\in \partial_{\bar{\epsilon}}f(\bar{x}),\text{ where }\bar{\epsilon}=\sum\limits_{i\in \bar{I}}\lambda_i \tilde{e}_i.
\end{equation}
The derivation of \eqref{con:modRed} and \eqref{con:epsSubd} can be found in Lemma \ref{reduction:lemma} where we prove the same conclusion for the generalization of LPBC, the LPBNC algorithm, with derivation following the same reasoning. The mapping $(x,\epsilon)\mapsto\partial_\epsilon f(x)$ is outer-semicontinuous, and hence when the model reduction decreases to 0 we have $0\in \partial f(x)$. Consequently, our stopping criterion is that the model reduction is sufficiently small as it is showed in line \ref{stop:conv} of Algorithm \ref{algo:1}. It is worthy to note that our model reduction in \eqref{con:modRed} is comparable with that in the classical bundle method which uses a quadratic model of the form 
\begin{equation}\label{mod:quad}
\min\limits_{x\in \mathbbm{R}^n} \quad\left[m(x)+\frac{\mu}{2}\normt{x-\bar{x}}^2\right].
\end{equation}
If the above model is used, then $m(\bar{x})-m(x^*)=\sum\limits_{i\in \bar{I}}\lambda_i \tilde{e}_i+\frac{1}{\mu}\normt{\sum\limits_{i\in \bar{I}}\lambda_i s_i}^2$. A significant difference between \eqref{transit} and \eqref{mod:quad} is that the latter is strictly convex but the former is not. The optimal solution to \eqref{mod:quad} is unique and can be expressed by $x^*=\bar{x}-\frac{1}{\mu}\sum\limits_{i\in \bar{I}}\lambda_i s_i$, while \eqref{transit} may have multiple solutions. The readers are referred to chapter XIV and XV of \cite{Hiriart-Urruty1993a} for a comprehensive understanding of classical bundle methods.
\LinesNumbered 
\begin{algorithm}
\caption{Algorithm LPBC}\label{algo:1}
\KwData{Final accuracy tolerance $\epsilon_{\mathrm{tol}}$,  and maximum trust region radius  $\Delta_{\max}$, initial trust region $\Delta^1_1\in[1,\Delta_{\max})$, initial point $x^0 $, trust region parameters $\eta_1,\ \eta_2, \ \eta_3$, integer $T\geq 20$ (inactive threshold); }
\textit{Initialization}\
Set the major and minor iteration counter $(k,l)\law (0,0)$\;
generate a cutting plane at $x^k$, update the index set $I(k,l)$ and define the cutting plane model \eqref{mod:con}. 
Set the minor iteration counter $l=0$, 
$y^k_1=x^k$ and compute $s^k_1\in\partial f(y^k_1)$\label{majorBegin}\;
Solve the linear programming subproblem \eqref{subpz2} and obtain an optimal solution $(x^{kl},z^{kl})$\label{line2z}\;
\label{stop:conv}\If{$f(x^k)- m^k_l(x^{kl})\leq (1+|f(x^k)|)\epsilon_{\rm tol}$}{STOP}\label{stop:end}
\eIf{$\rho_l^k=\frac{f(x^k)-f(x^{kl})}{f(x^k)-m^k_l(x^{kl})}\geq \eta_1$}{$x^{k+1}=x^{kl}$\;
obtain $\Delta^{k+1}_0$ via procedure \ref{btrutr}\;
obtain $m^{k+1}_0$ by keeping all cutting planes except those that have been inactive for $T$ iterations\;
$k=k+1$, continue to next major iteration by going to line \ref{majorBegin}}{obtain $\Delta^k_{l+1}$ via procedure \ref{btrutr}\;
delete all cutting planes except the one generated at $x^k$, 
those that are active at $x^{kl}$ and those that have been inactive for less than $T$ times. 
}
add the cutting plane
 $f(x^{kl})+{s^k_l}^T(x-x^{kl})$ to the model $m^k_{l+1}$\;
set $l=l+1$ and go to line \ref{line2z}\;
\end{algorithm}
Algorithm LPBNC is a generalization of LPBC and we will provide a convergence proof for this in section \ref{conana}. Hence we omit the proof here and only state the following convergence theorem for the convex case. 
\begin{theorem}
Suppose that $\epsilon_{\rm tol}=0$.\\
(i) If Algorithm \ref{algo:1} terminates at $x^{kl}$, then $x^k$ is a minimizer of $f$ with $x^k=P(x^k)$;\\
(ii) if there is an infinite number of minor iterations during the $k$th major iteration, then $x^k$ is a minimizer of $f$ with $x^k=P(x^k)$ and $\lim\limits_{l\rightarrow\infty}m^k_l(x^{kl})-f(x^k)=0$;\\
(iii) if the sequence of major iterations $\{x^k\}$ is infinite then $\lim\limits_{k\rightarrow\infty}\norm{x^k-P(x^k)} =0$.
\end{theorem}
\section{LP-bundle method : The nonconvex case}
\subsection{Derivation of the Method}
Based on the LPBC, we can derive a nonconvex version of the method through convexification for special types of functions. Specifically, we consider locally Lipschitz continuous functions that are prox-regular. Such functions are para-convex. 
Hence we can use a linear model to approximate a locally convex function. Under some assumptions, we show that the accumulation point of the minimizers of such functions is a stationary point of the objective function via the theory of proximal point mapping.

First we state the assumption on the objective function.
\begin{asump}\label{assump}
The objective function $f$ is locally Lipschitz continuous and bounded below. Given $x^0 \in\mathbbm{R}^n$ $f$ is prox-regular on bounded level set $\text{lev}_{ x^0}f$.
\end{asump}
Note that a single-valued function $f$ is prox-regular on an open set $O$ and locally Lipschitz continuous is equivalent to that $f$ is lower-$C^2$ on $O$; see \cite[Prop. 13.33]{Rockafellar1998}. 
Define 
\be\label{g} 
g(y)\mathrel{\mathop:}=g(y;x,a):y\mapsto f(y)+{\frac{a}{2}}||y-x||^2
\ee 
with $x$ and $a$ as parameters. Under Assumption \ref{assump}, the Moreau envelope function and the proximal point mapping associated with the objective $f$ are globally well defined. We redefine them here as 
\begin{align}
e_a(x)\mathrel{\mathop:}=\min\limits_{y\in \mathbbm{R}^n} \{g(y;x,a)\}\quad
P_{a}\left( x\right)\mathrel{\mathop:}=\arg\min\limits_{y\in \mathbbm{R}^n} \{g(y;x,a)\}.\label{pro:map}
\end{align}
If $f$ is para-convex, then according to Definition \ref{paraconvexity}, $g(y)$ is convex on some neighborhood $B_b(x)$. 
Clearly, for different $x$ and $b$, in order to make $g(y)$ convex with respect to $y$, there exists a threshold for the value of `$a$'. 
The motivation of our method is based on the following observation. 
Suppose we have some sequences $x^k\rightarrow x'$, $a^k\rightarrow a'$, and $b^k\rightarrow b'$ such that $g(y;x^k,a^k)$ is convex with respect to $y$ on $B_{b^k}(x^k)$ for all $k$. Then we can use a cutting-plane model for $g(y;x^k,{a^k})$ with trust region as in LPBC to obtain descent locally.  
To justify its stationarity, $x'$ should be a global minimizer of $g(y;x',a')$. In fact, the outer semicontinuity of the mapping $(x,a)\mapsto P_a(x)$ means if there exist $x^k\rightarrow\bar{x}$, $\{a^k\}$ bounded and $p^k\in P_{a^k}(x^k)$ with $\normt{x^k-p^k}\rightarrow 0$, then $\bar{x}\in P_{\bar{a}}(\bar{x})$ for some $\bar{a}$. 
Thus in order to find a stationary point our goal can be translated to generating a sequence $\{x^n\}$ and $\{p^n\}$ such that $\lim\limits_{n\rightarrow\infty}\norm{x^n-p^n}=0$ with $p^n\in P_{a^n}(x^n)$ and $\{a^n\}$ bounded.
Further discussion will be made in Section \ref{conana}.

The following lemma shows that there exists a threshold for the value $a$ such that the function $g(y)$ is locally convex on $\text{ lev}_{ x^0}f$ (which is not necessarily convex).
\begin{lemma}\label{convexific}
Under assumption \ref{assump}, there exists a number $\bar{a}\geq 0$ such that for any $a\geq \bar{a}$, the function $g(y;x,a)$ is convex on a neighborhood of $y$ for all $y\in\text{ lev}_{ x^0}f$.
\end{lemma}
\begin{proof}
According to Assumption \ref{assump}, $f$ is locally Lipschitz continuous and given $x^0\in\mathbbm{R}^n$, $f$ is prox-regular at each point in the compact set $\text{ lev}_{ x^0}f$. 
For all $x\in\text{ lev}_{ x^0}f$, there exist $\epsilon(x)$ and $a(x)$ such that $f$ is prox-regular at $x$ with respect to $\epsilon(x)$ and $a(x)$. 
By Proposition \ref{prox-para}, $f$ is para-convex on $B(x,\epsilon(x))$ with respect to $a(x)$ for all $x\in\text{ lev}_{ x^0}f$; i.e. the function $\tilde{g}(y;a(x))\mathrel{\mathop:}=f(y)+\frac{a(x)}{2}||y||^2$ is convex on $B(x,\epsilon(x))$ for each $x\in\text{ lev}_{ x^0}f$. We see $\left\{{\rm int }\ B(x,\epsilon(x))\big | x\in\text{ lev}_{ x^0}f\right\}$ is an open cover of $\text{ lev}_{ x^0}f$ and it has a finite subcover $\left\{{\rm int }\ B(x_i,\epsilon(x_i))\big|i=1,\cdots,m \right\}$ corresponding to some $a(x_i),\ i=1,\ \cdots,\ m$. Define $\bar{a}\mathrel{\mathop:}=\max\left\{a(x_i)\big|i=1,\cdots,m\right\}$. Then the function\\ $\tilde{g}(y;\bar{a})=\tilde{g}(y;a(x_i))+\frac{\bar{a}-a(x_i)}{2}||y||^2$ for all $i=1,\cdots,m$ is also convex on each $B(x_i,\epsilon(x_i))$.
Consequently, $g(y;x,\bar{a})=\tilde{g}(y;\bar{a})-\bar{a}\left\langle y,x\right\rangle+\frac{\bar{a}}{2}||x||^2$ is convex on each $B(x_i,\epsilon(x_i))$ and so is $g(y;x,a)=g(y;x,\bar{a})+\frac{a-\bar{a}}{2}\normt{y-x}^2$ for all $a\geq \bar{a}$.      
\end{proof}

The following theorem shows that there exists a threshold for the value $a$ such that the function $g(y)$ is the restriction to $\text{lev}_{ x^0}f$ of a convex function. See the Appendix for the proof of this theorem.
\begin{theorem}\label{conv:th}
Suppose $f$ is prox-regular and locally Lipschitz on a bounded level set $%
\text{lev}_{ x^0}f$ with $\text{int}$ $\text{lev}_{ x^0}f\neq \emptyset $.
Let $g\left( y;x,a\right)$ be defined in \eqref{g} with $a \geq 0$. 
There exists a number $a^{th}\geq 0$ such that for all $a\geq a^{th}$, $%
g(y;x,a)$ is the restriction to $\text{lev}_{ x^0}f$ of a globally convex function $H(y;x,a)$ satisfying $g(y;x,a)\geq H(y;x,a)$ for all $y\in\mathbbm{R}^n$ and $x\in \text{lev}_{x^0}f$. 
\end{theorem}
Theorem \ref{conv:th} essentially shows that $g$ can be described as a restriction of a convex function. We will in the future refer to this as the 'restriction property'.

\subsection{On-The-Fly Convexification}
The threshold value $a^{th}$ is hard to find. Our goal in this section (see also Section \ref{sec:updM}) is to find a lower bound for the parameter $a$ such that $g(y;\bar{x},a)$ is a restriction of a convex function locally within $\text{lev}_{x^0}f$. We first introduce the convexification technique which first appeared in \cite{Hare2010}. Suppose we are at the current iteration point, i.e. the current best candidate for a stationary point of $f$. We denote this point $\bar{x}$ in general and by $x^k$ when it is necessary to indicate it is in the $k$-th iteration. 
 A necessary condition for this is that all the cutting planes generated at the points in the subset should be below the graph of $g$, since a convex function is essentially represented by the point-wise supremum of cutting-plane functions.

Denote $\partial g(y;x,a) $ as the subdifferential of function $g$ with respect to variable $y$. 
It follows from the calculus of subdifferential and Assumption \ref{assump}
 that 
 \be 
 \partial g(y;x,a)=\partial f(y)+a(y-x) \text{ and } \partial g(x;x,a)=\partial f(x)
 \ee
  for all $y$ and $x$ in the set $\text{lev}_{x^0}f$. 
  For any $s\in \partial f(y)$ and $y\in \text{lev}_{x^0}f$, clearly we have $s+a(y-x)\in \partial g(y;x,a)$. 
  Consequently, the cutting-plane function of $g$ at the point $y_i$  
  can be written as 
\be\label{cuttingplaneFunc}
h(w,\bm{y}_i)\mathrel{\mathop:}= h(w;x,a,y_i):w\mapsto f(y_i)+{\frac{a}{2}}||y_i-x||^2+\la  s_i+a(y_i-x),w-y_i\ra, 
\ee
where $s_i\in \partial f(y_i)$. 
According to Theorem \ref{conv:th}, under Assumption \ref{assump}, if $a\geq a^{th}$, $g$ is a restriction to $\text{lev}_{x^0}f$ of a convex function $H$ minorizing $g$. Thus a cutting plane of $g$ generated at an auxiliary point $y_i\in\text{int }\text{lev}_{x^0}f$ is the same cutting plane of $H$ generated at $y_i$; additionally, as $H$ is a convex function minorizing $g$, this cutting plane is not only below the graph of $H$ but also below that of $g$. In summary we have 
\begin{equation}\label{con:c1}
g(y;{x},a)-h(y;{x},a,y_i)\geq 0,\ \forall\ y\in\mathbb{R}^n,\ x\in\text{lev}_{x^0}f,\ y_i\in\text{int lev}_{x^0}f,\  a\geq a^{th}.
\end{equation}
We provide a localized convexification process by selecting a collection of points around the current iteration point and verifying the necessary condition for convexity. We let `$a$' be variable and $I$ be some index set, and set 
\begin{equation}\label{con:con}
g(y_j;\bar{x},a)-h(y_j;\bar{x},a,y_i)\geq 0,\ \forall\ i,\ j\in I;
\end{equation}
to deduce the necessary condition for $a$: $a\geq \tilde{a}^{\min}$, where 
\begin{equation}\label{a:min}
\tilde{a}^{\min}\mathrel{\mathop:}= \max\limits_{i,j\in I}\{-\frac{f(y_i)-f(y_j)-\la s_{j},y_i-y_j\ra}{\frac{1}{2}\normt{y_i-y_j}^2}\}.
\end{equation}
This value can be negative, so we set ${a}^{\min}\mathrel{\mathop:}=\max\{\tilde{a}^{\min},0\}$. Consequently, for any $a\geq{a}^{\min}$, \eqref{con:con} holds true.  

Note that ${a}^{\min}$ is dependent on the points indexed in $I$. Consequently, each time a new auxiliary point $y_i$ is obtained $a^{\min}$ needs to be updated. We also note that, $a^{\min}$, the local lower bound for the convexification parameter, determined by \eqref{con:con} where $\bar{x}\in\text{lev}_{x^0}f,\ y_i\in\text{int lev}_{x^0}f,\ \forall\ i\in I$, is not greater than $a^{th}$ that satisfies \eqref{con:c1}. 

\subsection{The Model Problem and Model Reduction}
The cutting-planes model of $g(y;x,a)$ is defined by  
\be \label{model}
m(w)\mathrel{\mathop:} =m(w;x,a,I):w\mapsto \max\limits_{i\in I}\{h(w;x,a,y_i)\},\ee where $I$ is the index set of auxiliary points $y_i$ where cutting planes of function $g$ are generated. 
Our algorithmic model is defined by \eqref{model} and in the remainder of this paper, $m$ refers to the model defined in \eqref{model} unless otherwise stated. 
Suppose we are at $\bar{x}$. To proceed in finding a new candidate, we intend to obtain descent in $f$ by minimizing the cutting-planes model of $g(x;\bar{x},a)$ over a trust region, i.e. we solve the linear subproblem 
\be \label{subproblemoriginal}
\min\limits_{x\in\mathbbm{R}^n}\qquad m(x;\bar{x},a,I)\qquad \text{subject to }\quad \norm{x-\bar{x}}\leq\Delta,
\ee
which is equivalent to the following problem 
\begin{subequations}\label{subppc}
\begin{align}
\min\limits_{(x,z)\in \mathbbm{R}^{n+1}}&\qquad \qquad\qquad z\\
\mathrm{subject\ to}&\quad f(y_i)+{\frac{a}{2}}\normt{ y_i-\bar{x}}^2+\la s_i+a(y_i-\bar{x}),x-y_i\ra\leq z,\ i\in I,\label{sub2pc}\\
\  & \qquad\quad \norm{x-\bar{x}}\leq \Delta.\label{sub3pc}
\end{align}
\end{subequations}
We would like to inspect the reduction of the model function $m$ after we have obtained a new trial point via the linear programming problem \eqref{subppc}. 
Denote a general optimal solution of problem \eqref{subppc} by $(x^*,z^*)$ and by $(x^{kl},z^{kl})$ when it is necessary to indicate it is in the $l$-th minor iteration in the $k$-th major iteration. 
In the following lemma we derive the explicit expression of the reduction of the model from $\bar{x}$ to $x^*$. We will use the linearization errors of $g(\cdot;\bar{x},a)$ at $\bar{x}$:\\
\begin{align}
E_i\mathrel{\mathop:}=&g(\bar{x};\bar{x},a)-h(\bar{x};\bar{x},a,y_i)\label{Def:Ei:1}\\
=&f(\bar{x})-\left[f(y_i)+{\frac{a }{2}}\normt{y_i-\bar{x}}^2+\la s_i+a (y_i-\bar{x}),\bar{x} -y_i\ra\right],\ \forall\ i\in I.\label{Def:Ei}
\end{align}
\begin{lemma}\label{reduction:lemma}
Consider the linear problem \eqref{subppc}. Let $\bar{i}\in I$ be such that $\bar{x}=y_{\bar{i}}$, $m(x)\mathrel{\mathop:}= m(x;\bar{x},a,I)$ be defined as in \eqref{model}, $(x^*,z^*)$ be an optimal solution of \eqref{subppc} and suppose `$a$' is such that $E_i\geq 0,\ \forall \ i\in I$. \\(i) The following holds true
\be\label{modelreductionexpression} 
 m(\bar{x})-m(x^*) =f(\bar{x})-z^*=\begin{cases}
 \sum\limits_{i\in \bar{I}}\lambda_iE_i,&\text{if } \norm{x^*-\bar{x}}<\Delta;\\
\sum\limits_{i\in \bar{I}}\lambda_iE_i+\Delta
\normo{\sum\limits_{i\in \bar{I}}\lambda_i [s_i+a(y_i-\bar{x})]}, \mspace{-4mu} &\text{if } \norm{x^*-\bar{x}}=\Delta,
 \end{cases}
 \ee
where $\bar{I}$ is the index set for active constraints in \eqref{sub2pc} at $(x^*,z^*)$,  
and $\lambda_i$ for $i\in \bar{I}$ are the corresponding Lagrangian multipliers of \eqref{sub2pc}.
\\(ii) Let $C$ be any set satisfying $\bar{x}\in \mathrm{int}\ C$ and $y_i\in \mathrm{int}\ C$ for all $i\in \bar{I}$, if additionally `$a$' is such that $g(y;x,a)$ is a restriction to $C$ of a globally convex function $H(y;x,a)$ satisfying $g(y;x,a)\geq H(y;x,a)$ for all $y\in \mathbbm{R}^n$, $x\in\mathbbm{R}^n$.
Then 
\begin{equation}\label{conclu:2}
\sum\limits_{i\in \bar{I}}\lambda_i [s_i+a(y_i-\bar{x})]\in\partial_{\tilde{\epsilon}}g(\bar{x};\bar{x},a),\text{ where }\tilde{\epsilon}=\sum\limits_{i\in \bar{I}}\lambda_iE_i;
\end{equation}
and furthermore, if $0\in\partial_0g(\bar{x};\bar{x},a)$, then $\bar{x}$ is a global minimizer of $g(y;\bar{x},a)$.
\end{lemma}
\begin{proof}
(i) By definition $m(\bar{x})=\max\limits_{i\in I}\{h(\bar{x};\bar{x},a,y_i)\}\geq h(\bar{x};\bar{x},a,y_{\bar{i}})$ and since $\bar{x}=y_{\bar{i}}$, we have $h(\bar{x};\bar{x},a,y_{\bar{i}})=f(y_{\bar{i}})=f(\bar{x})$ by the definition of function $h$ in \eqref{cuttingplaneFunc}. As `$a$' is such that $E_i\geq 0$ for all $i\in I$, we get $\max\limits_{i\in I}\{h(\bar{x};\bar{x},a,y_i)\}\leq g(\bar{x};\bar{x},a)=f(\bar{x})$ via the definition of $E_i$ in \eqref{Def:Ei:1}. Consequently, $m(\bar{x})=f(\bar{x})$.
Problem \eqref{subppc} is equivalent to \eqref{subproblemoriginal} in the sense that the optimal solution $(x^*,z^*)$ satisfies $m(x^*)=z^*=\min \{m(x)\big |\norm{x-\bar{x}}\leq\Delta\}$. Therefore $m(\bar{x})-m(x^*) =f(\bar{x})-z^*$.
The $\bar{I}$ in \eqref{modelreductionexpression} is defined by 
\be \label{Ia}
\bar{I}\mathrel{\mathop:}=\left\{i\in I\big |f(y_i)+{\frac{a}{2}}\normt{y_i-\bar{x}}^2+\la s_i+a(y_i-\bar{x}),x^*-y_i\ra = z^*\right\}.
\ee
$\bar{I}$ cannot be empty because $(x^*,z^*)$ has to be on some cutting plane. 
If $\norm{x^*-\bar{x}}=\Delta$, then one of the sets
\be 
I_R\mathrel{\mathop:}=\left\{i\in\{1,\cdots,n\}\big|x^*_i=\bar{x}_i+\Delta\right\},\text{ and }I_L\mathrel{\mathop:}=\left\{i\in\{1,\cdots,n\}\big|x^*_i=\bar{x}_i-\Delta\right\}
\ee
will be nonempty.
As $(x^*,z^*)$ is the optimal solution of linear problem \eqref{subppc}, it satisfies the Karush-Kuhn-Tucker (KKT) conditions; 
that is, there exist multipliers $\lambda_i\geq0,\ i\in\bar{I},\  u_i\geq0,\ i\in I_R,\ \text{and}\ w_i\geq0,\ i\in I_L$ 
such that 
\be\label{KKT1}
\sum\limits_{i\in \bar{I}}\lambda_i[s_i+a(y_i-\bar{x})]+\sum\limits_{i\in I_R}u_ie_i-\sum\limits_{i\in I_L}w_ie_i=0,
\ee
\be\label{KKT2}
1-\sum\limits_{i\in \bar{I}}\lambda_i=0,\ee
where $e_i,\ i=1,\cdots,n$ are vectors in $\mathbbm{R}^n $ with the $i$-th component being 1 and the others 0.
If $\norm{x^*-\bar{x}}<\Delta$, then \eqref{KKT1} simply reduces to 
\be\label{KKT3}
\sum\limits_{i\in \bar{I}}\lambda_i[s_i+a (y_i-\bar{x})]=0.
\ee
We have
\begin{align}\label{derive}
f(\bar{x})-z^* &  =\sum\limits_{i\in \bar{I}}\lambda_i \left[ f(\bar{x})-z^* \right]  \text { (by \eqref{KKT2})}\quad\notag\\
 \ & =\sum\limits_{i\in \bar{I}}\lambda_i \left\{ f(\bar{x})-\left[f(y_i)+{\frac{a }{2}}\normt{y_i-\bar{x}}^2+\la s_i+a (y_i-\bar{x}),x^*-y_i\ra\right]\right\} \notag\text {( by } \eqref{Ia})\\
 &=\sum\limits_{i\in \bar{I}}\lambda_i \left\{f(\bar{x})-\left[f(y_i)+{\frac{a }{2}}\normt{y_i-\bar{x}}^2+\la s_i+a (y_i-\bar{x}),\bar{x} -y_i\ra\right]\right\} \notag\\
&\quad\quad\quad\quad\quad\quad\quad\quad\quad\quad\quad\quad\quad\  - \sum\limits_{i\in \bar{I}}\lambda_i\la s_i+a (y_i-\bar{x}),x^*-\bar{x} \ra.
\end{align}
Consider the first case when $ \norm{x^*-\bar{x}}<\Delta$. Then by \eqref{Def:Ei}, \eqref{KKT3} and \eqref{derive} we have \[f(\bar{x})-z^* =\sum\limits_{i\in \bar{I}}\lambda_iE_i.\]
If $\norm{x^*-\bar{x}}=\Delta$, then by \eqref{Def:Ei}, \eqref{KKT1} and \eqref{derive} we have
\begin{align}\label{derive1}
f(\bar{x})-z^* &  = \sum\limits_{i\in \bar{I}}\lambda_iE_i+\la\sum\limits_{i\in I_R}u_ie_i-\sum\limits_{i\in I_L}w_ie_i,x^*-\bar{x} \ra \notag\\
 & =\sum\limits_{i\in \bar{I}}\lambda_iE_i+\Delta\left(\sum\limits_{i\in I_R}u_i+\sum\limits_{i\in I_L}w_i\right)\ \text {(by definition of } I_R\text{ and }I_L).
\end{align}
On the other hand, 
\begin{align}\label{derive2}
\normo{\sum\limits_{i\in \bar{I}}\lambda_i[s_i+a(y_i-\bar{x})]} &=\normo{\sum\limits_{i\in I_R}u_ie_i-\sum\limits_{i\in I_L}w_ie_i}\text{ (by \eqref{KKT1}) } \notag\\
&=\left(\sum\limits_{i\in I_R}u_i+\sum\limits_{i\in I_L}w_i\right)\ \text {(by definition of } I_R\text{ and }I_L).
\end{align}
Combining \eqref{derive1} and \eqref{derive2} we get \[f(\bar{x})-z^* =\sum\limits_{i\in \bar{I}}\lambda_iE_i+\Delta\normo{\sum\limits_{i\in \bar{I}}\lambda_i[s_i+a(y_i-\bar{x})]}.\]

(ii) By the restriction property in Theorem \ref{conv:th} we have $g(y;\bar x ,a)=H(y;\bar{x},a)$ for all $y\in C$, and $\partial g(y;\bar x ,a)=\partial H(y;\bar{x},a)$ for all $y\in\mathrm{int}\ C$. 
As $H$ is globally convex, $\partial_{\epsilon}g(y;\bar{x},a)$ can be defined by $\partial_{\epsilon}g(y;\bar{x},a)\mathrel{\mathop:}=\partial_{\epsilon}H(y;\bar{x},a)$ for any $y\in \mathrm{int}\ C$.
For all $i\in \bar{I}$, since $y_i\in \mathrm{int}\ C$, we have $s_i+a(y_i-\bar{x})\in\partial H(y_i;\bar{x},a) $. Thus we have that
\begin{align}
g(z;\bar{x},a) & \geq H(z;\bar{x},a)\text{ for all }z\in\mathbbm{R}^n\notag\\
 & \geq  H(y_i;\bar{x},a)+ \la s_i+a(y_i-\bar{x}),z-y_i\ra\text{ for all }i\in \bar{I} \text{ and } z\in\mathbbm{R}^n\notag\\
 & \text{ (by definition of convex subgradient)}\notag\\
  & =g(y_i;\bar{x},a)+ \la s_i+a(y_i-\bar{x}),z-y_i\ra\text{ (by the restriction property)}\notag\\
  & =g(\bar{x};\bar{x},a)+\la s_i+a(y_i-\bar{x}),z-\bar{x}\ra- E_i  \text{ (by definition of }E_i).\label{for:epsilonSub}
\end{align}
The convex combination of \eqref{for:epsilonSub} with $\lambda_i$ satisfying \eqref{KKT2} yields 
\begin{equation}\label{interim}
g(z;\bar{x},a) \geq g(\bar{x};\bar{x},a)+\la \sum\limits_{i\in \bar{I}}\lambda_i [s_i+a(y_i-\bar{x})],z-\bar{x}\ra-\sum\limits_{i\in \bar{I}}\lambda_iE_i, \text{ for all }z\in\mathbbm{R}^n.
\end{equation}
By the definition of $\epsilon-$subdifferential in convex analysis, \eqref{interim} verifies \eqref{conclu:2}.

If $0\in \partial_0g(\bar{x};\bar{x},a)$, then $0\in \partial H(\bar{x};\bar{x},a)$ since $\bar{x}\in \mathrm{int}\ C$ and $\partial H(y;\bar{x},a)=\partial g(y;\bar{x},a)$ for all $y\in \mathrm{int} \ C$. Because $H(y;\bar{x},a)$ is a convex function, its stationary points are also global minimizers. Consequently, for any $z\in\mathbbm{R}^n$, we have $g(\bar{x};\bar{x},a)=H(\bar{x};\bar{x},a)\leq H(z;\bar{x},a)\leq g(z;\bar{x},a)$. We note that another simple way to see this conclusion is that $0\in\partial g(\bar{x};\bar{x},a)=\partial f(\bar{x})$ which is equivalent to $\bar{x}\in P_a(\bar{x})$.
\end{proof}
\begin{remark}
	
(i) 
The conclusions in Lemma \ref{reduction:lemma} are very similar to those of the classical bundle methods. The convex combination of subgradients, $\sum\limits_{i\in \bar{I}}\lambda_is_i$ (or $\sum\limits_{i\in \bar{I}}\lambda_i\left[s_i+a(y_i-\bar{x})\right]$) is involved in both situations and is sometimes termed an \emph{aggregate subgradient}. The \emph{subgradient aggregation} technique was developed in \cite{kiwiel1985methods} where aggregate subgradients together with the convex combinations of linearization errors are used to represent additional virtual cutting planes in the model. Subgradient aggregation technique has many applications including preventing unbounded storage caused by too many cutting planes. 
\\

(ii) It is not difficult to see that the model reduction is always nonnegative provided that $E_i\geq 0$ for all $i
\in {I}$. From the definition of $E_i$ in \eqref{Def:Ei:1}, this can be guaranteed by choosing $a\geq a^{\min}$ which satisfies \eqref{con:con}.

(iii) The expression of model reduction in \eqref{modelreductionexpression}  can also be stated as 
\begin{equation}
m(\bar{x})-m(x^*)=\sum\limits_{i\in \bar{I}}\lambda_iE_i+\Delta
\normo{\sum\limits_{i\in \bar{I}}\lambda_i [s_i+a(y_i-\bar{x})]}
\end{equation}
with $\sum\limits_{i\in \bar{I}}\lambda_i [s_i+a(y_i-\bar{x})]=0$ if $\norm{x^*-\bar{x}}<\Delta$. 
\end{remark}
Lemma \ref{reduction:lemma} implies that the model reduction can somehow help us to determine whether $\bar{x}$ is a good estimate of a stationary point. 
If $f$ is convex, generally speaking, a good estimate $\bar{x}$ of a minimizer of  $f$ should satisfy that both $\min\limits_{g\in\partial f(\bar{x})} ||g||$ and $f(\bar{x})-f(p)$ are very small, where $p$ is a minimizer of $f$. 
In the convex case $\frac{f(\bar{x})-f(p)}{\norm{\bar{x}-p}}$ can be a lower bound for $\min\limits_{g\in\partial f(\bar{x})} ||g||$. Motivated by this, in the next lemma we try to relate the model reduction with the above two approximate measures of a good estimate of a stationary point in the nonconvex case through the restricted convexity. 
We will use the following set which essentially defines the model $m(x)$.
\begin{equation}\label{Fset}
F\mathrel{\mathop:}=\{y_i\ |\ i\in \hat{I}\}, \text{ where }\hat{I}\mathrel{\mathop:}=\{i\in I\ | \exists\ x\in\mathbbm{R}^n\ \text{such that} \ m(x)=h(x;\bar{x},a,y_i)\}.
\end{equation}
\begin{lemma}\label{local:mod}
Given a prox-center $\bar{x}\in\text{lev}_{x^0}f$, a trust region radius $\Delta$ and an index set $I$ containing some $\bar{i}$ such that $\bar{x}=y_{\bar{i}}$. 
Let $m(x)$ defined in \eqref{model} be the objective function of problem \eqref{subproblemoriginal}, `$a$' be such that $E_i\geq 0,\ \forall \ i\in I$, $x^*$ be the first component of an optimal solution of problem \eqref{subppc}, and $F$ be defined in \eqref{Fset}.

If `$a$'  is such that $g(x;\bar{x},a)$ is a restriction to set $F$ of $H(x)$, where $H(x)$ is convex on $\mathbbm{R}^n$ and satisfies $H(x)\leq g(x)$ for all $x\in\mathbbm{R}^n$.
Then \\
(a) $m(x)$ is a cutting-plane model of $H(x)$ and satisfies $m(x)\leq H(x),\ \forall\ x\in\mathbbm{R}^n$;\\
(b) if $p\in P_{a}(\bar{x})$ and $\bar{x}\not\in P_{a}(\bar{x})$, then  
\begin{equation}\label{mod:rd1}
m(\bar{x})-m(x^*)\geq [f(\bar{x})-e_{a}(\bar{x})]\min\{\frac{\Delta}{\norm{\bar{x}-p}},1\}\geq 0.
\end{equation}
\end{lemma}
\begin{proof}
(a) We see that $\hat{I} $ indexes all the cutting planes that sufficiently define $m(x)$. The model $m(x)$ is essentially the pointwise maximum of cutting planes of $g(x)$ generated at bundle points where the value of $g$ and $H$ coincide, and hence $m(x)$ is also a cutting-plane model of $H(x)$. Since $H(x)$ is convex, by proposition \ref{con:prop} it is a lower approximation of $H$. This finishes the proof of (a).\\
(b) The proof of \eqref{mod:rd1} can be divided into two parts based on the possible positions of $p$. 
First, suppose $p$ is located in the trust region, i.e. $\norm{\bar{x}-p}\leq \Delta$, which yields $\min\{\frac{\Delta}{\norm{\bar{x}-p}},1\}=1$. To show \eqref{mod:rd1} it suffices to show $m(\bar{x})-m(x^*)\geq f(\bar{x})-e_{a}(\bar{x})$.   
By Lemma \ref{reduction:lemma} (i), $m(\bar{x})=f(\bar{x})$, and hence we only need to show $m(x^*)\leq e_{a}(\bar{x})$. From the optimality of $x^*$, conclusion (a), the fact that $H(x)\leq g(x)$ and \ref{pro:map}, we have 
\[m(x^*)\leq m(p)\leq H(p)\leq g(p)= e_{a}(\bar{x}).\]
Second, suppose $p$ is outside the trust region, i.e. $\norm{\bar{x}-p}>\Delta$, which yields $\min\{\frac{\Delta}{\norm{\bar{x}-p}},1\}=\frac{\Delta}{\norm{\bar{x}-p}}$. To show \eqref{mod:rd1} it suffices to show
\begin{equation}\label{mid:mod2}
m(\bar{x})-m(x^*)\geq \frac{f(\bar{x})-e_{a}(\bar{x})}{\norm{\bar{x}-p}}\Delta. 
\end{equation}
We consider the point $x^c=\bar{x} +\frac{\Delta}{\norm{\bar{x}-p}}(p-\bar{x})$, the  intersection point of trust region and the line segment $[\bar{x},p]$. By the optimality of $x^*$, the result (a), the convexity of $H$, the fact that $H(x)\leq g(x)$, (\ref{pro:map}) and the fact that $g(\bar{x})=f(\bar{x})$, we have
\begin{eqnarray*}
m(x^*)\leq m(x^c)\leq H(x^c)\leq \frac{\Delta}{\norm{\bar{x}-p}}H(p)+(1-\frac{\Delta}{\norm{\bar{x}-p}})H(\bar{x})\\
\leq \frac{\Delta}{\norm{\bar{x}-p}}g(p)+(1-\frac{\Delta}{\norm{\bar{x}-p}})g(\bar{x})\\
= \frac{\Delta}{\norm{\bar{x}-p}}e_{a}(\bar{x})+(1-\frac{\Delta}{\norm{\bar{x}-p}})f(\bar{x}).
\end{eqnarray*}
Then \eqref{mid:mod2} can be verified using $f(\bar{x})=m(\bar{x})$.
\end{proof} 

\subsection{Update of the Model}\label{sec:updM} 
At the end of a certain iteration (either major or minor), we 
need to update the model and prepare the data for the new LP in next iteration. 
The update of the model is supposed to improve the model. 
The update includes adding new cutting planes and deleting old cutting planes, i.e. adding and removing points from the set $\Omega\mathrel{\mathop:}=\{y_i\ |\ i\in I\}$. We also take into account the update of convexification parameter $a$ and $a^{\min}$ when considering updating the model, as $a$ is also part of the model. 

In our method, we always add one cutting plane at the end of each iteration. Specifically, at the end of a major iteration, we obtain $x^{k+1}$ as our new prox-center. A cutting plane will be needed to generate at this point, and hence we add $x^{k+1} $ to $\Omega$ so that there exists a $\bar{i}\in I$ such that $y_{\bar{i}}=x^{k+1}$. At the end of a minor iteration, we obtain $x^{kl}$ which did not yield sufficient reduction of the objective function. A cutting plane is supposed to be generated at this point to improve the quality of the model. However $x^{kl}$ could be very bad in the sense that $f(x^{kl})$ is too far away from $f(x^k)$ or even bigger than $f(x^0)$. In this case we backtrack along the direction $x^{kl}-x^k$ until we find a point whose funciton value is less than some upper bound $f^k_u$. This is a finite process provided that $f^k_u>f^k$, as we will prove in Lemma \ref{backtrack:lemma}.

After backtrack, we add the point found into $\Omega$, and consequently, all our new bundle points, i.e. those that are generated in iteration $(k,l)$ for some $l$, will be in lev$_{f^k_u}f$. 
However the old bundle points, i.e. those generated in $(k-j,l')$ for some $j$ and $l'$, can still be outside lev$_{f^k_u}f$. At the end of iteration $(k,l_k)$, we remove 
the old bundle points whose function values are greater than $f^k_u$. Finally, before we enter iteration $k+1$ we move the upper bound closer to $f(x^{k+1})$ by setting $f^{k+1}_u\law \alpha_3f^k+(1-\alpha_3)f^k_u$ with some $\alpha_3\in (0,1)$.   

We see the bundle point set $\Omega$ is updated dynamically so that at any iteration $(k,l)$, $\Omega\subseteq \text{lev}_{f^{k-2}_u}f$. 
This setting is related to the convexification process. In Lemma \ref{local:mod} the value of parameter $a$ such that $g$ is a restriction of a convex function depends on the set $F\subseteq\Omega$. The following lemma states the existence of such value.  
\begin{lemma}\label{lm:34}
Suppose that Assumption \ref{assump} holds. 
Given a prox-center $\bar{x}\in\text{lev}_{x^0}f$ and some $f_u\in(f(\bar{x}),f(x^0)] $, consider the model $m(x)$ defined in \eqref{model} with bundle points $y_i\in \Omega$ satisfying $y_i\in\text{lev}_{f_u}f$ for all $i\in I$. 
Let $D$ be a                                                                   compact set such that $D\supseteq F$ and $\text{int }D\not=\emptyset$.
There exists a threshold $a^{th}(\bar{x},I)\geq 0$ such that for all $a\geq a^{th}(\bar{x},I)$, $%
g(y;x,a)$ is the restriction to $D$ of a globally convex function $H(y;x,a)$ satisfying $g(y;x,a)\geq H(y;x,a)$ for all $y\in\mathbbm{R}^n$ and $x\in D$.
\end{lemma}
\begin{proof}
This lemma is an extension of Theorem \ref{conv:th}. Since $y_i\in\text{lev}_{f_u}f$ for all $i\in I$ and $f_u\in(f(\bar{x}),f(x^0)] $ we have by \eqref{Fset} that $F\subseteq \text{lev}_{x^0}f$. As $F$ is a finite set and $D$ is the smallest compact set containing $F$ we have $D\subseteq \text{lev}_{x^0}f$. 
From Corollary \ref{cor:convex} and the poof of Theorem \ref{conv:th} in the appendix, we can see that the same conclusion holds true when we replace $\text{lev}_{x^0}f$ by any of its compact subset that has nonempty interior. In our case each $(\bar{x}, I)$ corresponds to a set $D\subseteq \text{lev}_{x^0}f$ and for each $D$ there exists a threshold $a^{th}(\bar{x},I) $ satisfying the 
corresponding conditions.
\end{proof}
From Lemma \ref{lm:34} we see that the condition for `$a$' in Lemma \ref{local:mod} can be satisfied if we take $a\geq a^{th}(\bar{x},I)$.
\subsection{The LPBNC Algorithm}
For the trust region update we follow the procedure \ref{btrutr}.  In our algorithm  in order to distinguish the prox-center we differentiate major iterations and minor iterations. 
When we set a new prox-center $x^{k+1}$ it is also the last minor iteration point denoted by $x^{kl_k}$. As we have an infinite sequence of iterations, the following two situations can happen. 
First, there are infinite number of major iterations with each loop of minor iteration to be finite. The sequence can be described as follows:
\begin{subequations}\label{itersequence:1}
\begin{align}
x^0 ,\ x^{00},\ x^{01},\cdots,\ x^1(=x^{0l_0}),\ x^{10},\ x^{11},\ \cdots,\ x^k,\ x^{k0},\ x^{k1},\ \cdots,\ x^{k+1}(=x^{kl_k}),\ \cdots\cdots.\label{itersequence:1A}
\end{align}
\end{subequations}
Second, there are finite number of major iterations with the last major iteration containing infinite minor iterations, which can be described as:
\be
x^0 ,\ x^{00},\ x^{01},\cdots,\ x^1(=x^{0l_0}),\ x^{10},\ x^{11},\ \cdots,\ x^k,\ x^{k0},\ x^{k1},\ \cdots,\ x^{kl},\ x^{k(l+1)},\ \cdots\cdots.\tag{\ref{itersequence:1}b}
\ee
A set of similar notations goes for the sequences of parameters $(a,\Delta)$. Starting from $(a^0_0,\Delta^0_0)$, each pair $(a^k_l,\Delta^k_l)$ is used to produce $x^{kl}$; and if $l=l_k$ we say $x^{kl_k}$ was produced by $(a^k_{l_k},\Delta^k_{l_k})$. To alleviate notation we drop the subscripts of $(a^k_l,\Delta^k_l)$ in the Algorithm \ref{LPBNC3} below and also in later analysis we define $m^k_l(x)\mathrel{\mathop:}=m(x;x^k,a^k_l,I(k,l))$ for all $k$ and $l$ and note the value of $m^k_l(x)$ is dependent on $x^k,a^k_l$ and $I(k,l)$.

\LinesNumbered
\begin{algorithm}
\label{LPBNC3}
\caption{LPBNC
}
\KwData{Final accuracy tolerance $\epsilon_{\mathrm{tol}}$, maximum trust region radius  $\Delta_{\max}$, initial trust region $\Delta^0_0\in(0,\Delta_{\max})$, 
initial point $x^0 $, trust region parameters $0<\eta_1<\eta_3<1$ and $0<\alpha_1<1<\alpha_2$, backtrack parameter $\beta\in(0,1)$, 
parameter 
$\sigma\in\mathbb{R}_{\geq 1}$ and increasing parameter for convexification parameter $\gamma\in[2,10]$.}
\textbf{Initialization} major and minor iteration counter $(k,l)\law (0,0)$, initial convexification parameter $a\law0$, $a^{\min}\law 0$, add $x^k$ into $\Omega$, generate a cutting plane of $g(y;x^k,a) $ at $x^k$, prepare the information for the first LP, $f^0_u \law f(x^0)$\;    
solve the linear programming subproblem \eqref{subppc} with $\bar{x},\ I$ replaced by $x^k,\ I(k,l)$ and obtain an optimal solution $(x^{kl},z^{kl} )$\label{line2z}\;
\If{$f(x^k)-z^{kl}\leq(1+|f(x^k)|)\epsilon_{\rm tol}$}{STOP; 
}
\eIf{$\rho_l^k=\frac{f(x^k)-f(x^{kl})}{f(x^k)-z^{kl}}\geq \eta_1$\label{rho:serious2}}
{$serious\law1$, 
$x^{k+1}\law x^{kl}$, $l_k\law l$\label{set:xkplus12}\;
\tcc{update trust region radius for major iteration}
\If
{$\rho^k_l>\eta_3$ \textbf{ and }$\norm{x^{kl}-x^k}>0.9\Delta $}{$\Delta\law \min\{\alpha_2\Delta ,\Delta_{\max}\}$\label{increase:Del}}
for all $w\in\Omega$ if $w\not\in\text{lev}_{f_u ^k}f$, delete $w$ from $\Omega$ and update the index set $I(k,l)$ by deleting the index whose corresponding cutting plane is generated at $w$
\label{del:cuts}\;
$f_u ^{k+1}\law \alpha_3 f(x^{k+1})+(1-\alpha_3)f_u ^k$
}
{$serious\law 0$\;
\tcc{update trust region radius for minor  iteration}
\If{$ \rho^k_l<-\frac{1}{\min\{1,\Delta \}}$\label{judgerho}}{$\Delta\law\alpha_1 \Delta $\label{decr:del}}
\If (\tcp*[h]{backtrack 
}){$k>0$ and $f(x^{kl})>f_u ^k$}
	{$\bm{d}\law x^{kl}-x^k$;$j\law 1$\;
		\While{$f(x^k+\beta^j\bm{d})>f_u ^k$}
			{$j\law j+1$}
	$\bm{\bar{y}}\law x^k+\beta^j\bm{d}$}
}
add $x^{k+1}$ in major iteration, or add $x^{kl}$, or $\bar{y}$ if backtrack was performed, as a new bundle point $y_i$ into $\Omega$,  update $a^{\min}$ according to its definition in \eqref{a:min}\;
\uIf{ $a<a^{\min}$}
{$a\law \max\left\{a^{\min},\gamma a \right\}$\label{updateA3}}
\ElseIf{$a^{\min}>0$ \textbf{and }$a\geq \sigma a^{\min}$}{$a\law(a+a^{\min})/2$\label{updateA4}}
\eIf{serious}{generate a cutting plane of $g(y;x^{k+1},a) $ at $x^{k+1}$, and add the cutting-plane function to the model. (The model becomes $m^{k+1}_0$ and the cutting plane index set becomes $I(k+1,l)$.)
$k\law k+1$, $l\law 0$\;
}{generate a cutting plane of $g(y;x^{k},a) $ at the new bundle point, and add the cutting-plane function to the model. (The model becomes $m^{k}_{l+1}$ and the cutting plane index set becomes $I(k,l+1)$.)

 $l\law l+1$}
  update the coefficient matrix in LP subproblem \eqref{subppc}\;
 continue to next iteration by going to line \ref{line2z}
\end{algorithm}
\newpage
\begin{remark}\label{f:dec}
\begin{enumerate}
\item The value of $a$ used in the LP subproblem \eqref{subppc} can always guarantee $E_i\geq 0$ for all $i\in I(k,l)$.
\item \label{rem2}We set $x^{k+1}$ as $x^{kl}$ if $\rho^k_l\geq \eta_1$. 
This implies $f(x^{k+1})<f(x^k)$ and thus $x^k\in\text{lev}_{x^0}f$ for all $k$. 
\item At the beginning of iteration $k$, a cutting plane of $g(\cdot;x^k,a)$ is generated at $x^k$. From the update of $f^k_u$ we see that $f(x^k)\leq f^k_u$ for all $k$. At the end of iteration $k$, $x^k$ will not be deleted from $\Omega$ and consequently $x^k$ is always indexed in $I(k,l)$ for all $0\leq l\leq l_k$.  
\end{enumerate}
\end{remark}
\subsection{LPBNC is well-defined}

The following lemma shows that if $x^{kl}$ is not located in $\text{lev}_{f^k_u}f$ then after finite backtrack along the direction $x^{kl}-x^k$, we can reach an auxiliary point in $\text{lev}_{f^k_u}f$. 
\begin{lemma}\label{backtrack:lemma}
If at iteration $(k,l)$ we have $f(x^{kl})>f^k_u$ and $k>0$, then there exists an integer $j'$ such that $\bar{y}=x^k+\beta^{j'}\bm{d}$ and  $\bar{y}\in\text{lev}_{f^k_u}f$ with $\bm{d}=x^{kl}-x^k$.  
\end{lemma}
\begin{proof}
From the algorithm we see $f^0_u=f(x^0)>f(x^k)$ and $f^k_u=\alpha_3f(x^k)+(1-\alpha_3)f^{k-1}_u$ if $k>0$. Hence, $f^k_u-f(x^k)>0$ for all $k>0$. 
Suppose for contradiction that $f(x^k+\beta^j\bm{d})>f^k_u$ for all integer $j>0$.  Hence, $0<f^k_u-f(x^k)<f(x^k+\beta^j\bm{d})-f(x^k)\leq\beta^j\norm{\bm{d}}\bar{L}$ where $\bar{L}$ is the Lipschitz constant of $f$. Therefore $\beta^j>\frac{f^k_u-f(x^k)}{\norm{\bm{d}}\bar{L}}\mathrel{\mathop:}=c_1$. Since $\beta\in(0,1)$ clearly this cannot be true for all $j>0$; a contradiction.
\end{proof}
We want to show that the algorithm LPBNC is well defined by showing that the inner loop (loop of minor iterations) can terminate finitely and that if it does not terminate finitely, then we have already found a stationary point of $f$.  
\begin{lemma}\label{modnondecr}
Suppose $\rho^k_l<\eta_1$ for some $k,\ l$. Then 
\[m^k_{l+1}(x^{k,l+1})\geq m^k_l(x^{kl}) \]
\end{lemma}
\begin{proof}
At the end of iteration $k,l$, we do not delete any cutting planes but add a new cutting plane to the model. Furthermore, the trust region radius $\Delta$ is possibly decreased. Hence the feasible region of linear subproblem \eqref{subppc} for iteration $k,l+1$ will become smaller. Therefore we have $m^k_{l+1}(x^{k,l+1})\geq m^k_l(x^{kl})$. 
\end{proof}
Define 
\begin{equation}\label{a}
L=\sup\{||s||_1:s\in\partial f(y),\ ||y-x||_{\infty}\leq\Delta_{\rm max},\ x\in  \text{lev}_{x^0 }f\}
\end{equation}
From the Lipschitz continuity and prox-regularity of $f$ we have $L<+\infty$.

The notion of \textit{minor iteration} is similar to the \textit{null step} in bundle methods. The following lemma shows that minor iterations either terminate finitely or generate an infinite sequence with very small model reduction. Furthermore, as we can see from below, the model reduction eventually decreases to 0 if minor iterations do not terminate finitely. If $f$ is convex we can easily see that this will show that the current iteration point is already a global minimizer of $f$. For the nonconvex case,  we will show that during the infinite minor iterations, if the convexification succeeds, i.e. the function $g$ is eventually convex locally around $x^k$, then $x^k$ is a stationary point of $f$.
\begin{lemma}\label{MinorTermiOr}
Suppose that $\epsilon_{\rm tol}=0$ and $a^k_l$ is bounded above by a constant $A$ for all $k$ and $l$. Let $l_1$ be any index such that $\rho^k_{l_1}<\eta_1$. Then there is an index $l_2>l_1$ and a real number $\bar{\eta}_2\in(\eta_1,1)$ such that either $\rho^k_{l_2}\geq\eta_1$ or 
\begin{equation}\label{mup2}
\frac{f(x^k)-m^k_{l_2}(x^{kl_2})}{f(x^k)-m^k_{l_1}(x^{kl_1})}\leq\bar{\eta}_2.
\end{equation}
\end{lemma}
\begin{proof}
Suppose for contradiction that there does not exist such index $l_2$ and real number $\bar{\eta_2}$; that is, there is an infinite sequence of minor iterations and 
\begin{equation}\label{contra1}
\frac{f(x^k)-m^k_{q}(x^{kq})}{f(x^k)-m^k_{l_1}(x^{kl_1})}>\bar{\eta}_2, \quad\forall\ q>l_1,\ \forall\ \bar{\eta}_2\in(\eta_1,1).
\end{equation} 
Since $l_1$ can be any index such that $\rho^k_{l_1}<\eta_1$, 
and we do not delete cutting planes in minor iterations, 
we can assume that $q$ and $l$ are generic indices satisfying $q>l\geq l_1$. 

To construct a contradiction, write $f(x^k)-m^k_{q}(x)=[f(x^k)-m^k_{q}(x^{kl})]+[m^k_{q}(x^{kl})-m^k_{q}(x)]$. 
Consider the two parts of the right hand side of this equation. 
First, observing that $x^{kl}$ is the point where a 
new cutting plane of $g(y;x^k,a^k_l)$ is generated 
and all the cutting planes in minor iterations are 
kept in the model, it follows that 
$g(x^{kl};x^k,a^k_q)\leq m^k_q(x^{kl})$. Consequently, 
\vspace{-0.5cm}
\begin{align}
f(x^k)-m^k_{q}(x^{kl})&  \leq f(x^k)-g(x^{kl};x^k,a^k_q)=f(x^k)-\left[f(x^{kl})+\frac{a^k_q}{2}||x^{kl}-x^k||^2\right]\notag
\\
\quad\quad & \leq f(x^k)-f(x^{kl})< \eta_1\left[f(x^k)-m^k_{l}(x^{kl})\right]\quad(\text{becasue}\ \rho^k_l<\eta_1)
\notag
\\
\label{temp1}\quad\quad &\leq \eta_1\left[f(x^k)-m^k_{l_1}(x^{kl_1})\right]\quad(\text{by Lemma \ref{modnondecr}}).
\end{align}
Second, the model function $m^k_q(x)$ is convex. 
Therefore for all $\tilde{s}\in \partial m^k_q(x^{kl})$, 
\begin{equation}\label{subgra}
m^k_q(x)-m^k_q(x^{kl})\geq \tilde{s}^T(x-x^{kl})\quad\forall\ x.
\end{equation}
Note $\tilde{s}$ is a subgradient of $m^k_q$ at $x^{kl}$ where a cutting plane of $g(x^{kl};x^k,a^k_q)$ was generated, and thus $\tilde{s}\in\partial g(x^{kl};x^k,a^k_q)=\partial f(x^{kl})+a^k_q(x^{kl}-x^k) $.
It follows from \eqref{a}, \eqref{sub3pc} and the boundedness of $a^k_l$ that 
$||\tilde{s}||_1\leq L+||a^k_q(x^{kl}-x^k)||_1\leq L+An\Delta_{\max}$. Applying this to \eqref{subgra} we have 
\be
\label{temp2}
m^k_q(x^{kl})-m^k_q(x)\leq ||\tilde{s}||_1||x-x^{kl}||_{\infty}\leq (L+An\Delta_{\max})||x-x^{kl}||_{\infty}\quad\forall\ x.
\ee
Summing \eqref{temp1} and \eqref{temp2}, there is 
\be\label{t3}
f(x^k)-m^k_{q}(x)<\eta_1 [f(x^k)-m^k_{l_1}(x^{kl_1})]+(L+An\Delta_{\max})||x-x^{kl}||_{\infty}\quad\forall\ x.
\ee
It then follows from \eqref{contra1} and \eqref{t3} by taking $x$ as $x^{kq}$ that 
\[\bar{\eta}_2[ f(x^k)-m^k_{l_1}(x^{kl_1})]< \eta_1 [f(x^k)-m^k_{l_1}(x^{kl_1})]+(L+An\Delta_{\max})||x^{kq}-x^{kl}||_{\infty}\quad\forall\ q>l_1.\]
Thus 
\be
||x^{kq}-x^{kl}||_{\infty}>\frac{\bar{\eta}_2-\eta_1}{(L+An\Delta_{\max})}[f(x^k)-m^k_{l_1}(x^{kl_1})]\mathrel{\mathop:}=c_2>0.
\ee
However this cannot happen for an infinite number of indices $q$ and $l$ because all minor iteration points $x^{kl}$ such that $l\geq l_1$ are in the neighborhood of $x^k$, $B(x^k,\Delta_{\rm max})=\{x \; | \; ||x-x^k||_{\infty}\leq \Delta_{\max}\}$. Hence a contradiction is found and \eqref{mup2} must be true for some $l_2> l_1$ and $\bar{\eta}_2>\eta_1$.
\end{proof} 
It is worthy to mention that the proof of lemma \ref{modnondecr} and \ref{MinorTermiOr} does not require the convexity of $f$. So far we have shown that the minor iterations either terminate finitely or continue infinitely. To demonstrate our algorithm is well defined, we need to show in the latter case that the current major iteration point is a stationary point of $f$ and that $a^k_l$ is indeed bounded above. We show this in the next section.
\section{Convergence analysis}\label{conana}
\begin{theorem}\label{convergenceth}
Let Assumption 1 hold and $\epsilon_{\textrm tol}=0$. Suppose Algorithm \ref{LPBNC3} terminates at iteration $(k,l)$. If $a^k_l\geq a^{th}(x^k,I(k,l))$ then $x^k\in P_{{a^k_l}}(x^k)$ and $0\in\partial f(x^k)$. 
\end{theorem}
\begin{proof}
As the algorithm terminates at $x^{kl}$, $x^{kl}$ must satisfy the stopping criterion $f(x^k)-z^{kl}\leq (1+|f(x^k)|)\epsilon_{\rm tol}$. 
As $\epsilon_{\rm tol}=0$, we have $f(x^k)-z^{kl}\leq0$.
However by the expression of model reduction in
 \eqref{modelreductionexpression}, 
 $f(x^k)-z^{kl}\geq0$. Thus $f(x^k)-z^{kl}=0$. Suppose for contradiction $x^k\not\in P_{{a^k_l}}(x^k)$. 
By \eqref{modelreductionexpression}, Lemma \ref{lm:34}, and Lemma \ref{local:mod}, if $a^k_l\geq a^{th}(x^k,I(k,l))$ and $p\in P_{{a^k_l}}(x^k)$, then we have 
\begin{equation}\label{mid:1}
0=f(x^k)-z^{kl}\geq \left[f(x^k)-e_{{a^k_l}}(x^k)\right]\min\left(\frac{\Delta^k_l}{\norm{x^k-p}},1\right)\geq 0.
\end{equation} 
From the definition of $P_{a^k_l}(x^k)\ni p$ we have 
\be\label{ggreterf}
f(x^k)-e_{{a^k_l}}(x^k)=g(x^k;x^k,a^k_l)-g\left(p;x^k,a^k_l\right)>0,\ \text{if }x^k\not
\in P_{a^k_l}(x^k).
\ee 
From \eqref{mid:1} and \eqref{ggreterf} we see $0=\min\left\{\frac{\Delta^k_l}{\norm{x^k-p}},1\right\}$. Thus
 $0=\frac{\Delta^k_l}{\norm{x^k-p}}$ with $\norm{x^k-p}\geq \Delta^k_l$. But $\Delta^k_l$ cannot be reduced to $0$ after finite iterations. Hence we have a contradiction.
\end{proof}
Let $\epsilon_{\rm tol}$ be 0. In the following analysis we assume the algorithm does not stop finitely. We see in \eqref{itersequence:1} that two sequences can be generated. Consistent notations should be used for $a$ and $a^{\min}$. For clear understanding, we unify those two cases with $\{a^n\}$ and $\{a^{\min}_n\}$ when it's not necessary to distinguish them.
We start our convergence analysis by showing that the convexification parameter $a$ in algorithm LPBNC is bounded.
\begin{lemma}\label{lem:aBound}
$0\leq a^n\leq\gamma a^{th},\ \forall\ n\in \mathbb{N}$.
\end{lemma}
\begin{proof}
The update of $a^n$ happens only in line \ref{updateA3} or line \ref{updateA4} of Algorithm \ref{LPBNC3}. In either case we have $a^n\geq a^{\min}_n\geq 0$ from the definition of $a^{\min}_n$. We increase $a$ in line \ref{updateA3} and decrease $a$ in line \ref{updateA4}. To show $a^n\leq\gamma a^{th}$ we only need to show $\max\{a^{\min}_n,\gamma a^n\}\leq \gamma a^{th}$ if $a^n<a^{\min}_n$. As $\gamma\geq 2$ and $a^{\min}_n\leq a^{th}$ for all $n\in \mathbb{N}$, we clearly have $\max\{a^{\min}_n,\gamma a^n\}\leq \gamma a^{\min}_n\leq\gamma a^{th}$.
\end{proof}
A consequence of lemma \ref{MinorTermiOr} is that if the minor iteration sequence does not terminate finitely then the model reduction will become smaller and smaller and eventually converge to $0$. We will show that if there is an infinite sequence of serious steps, the model reduction will converge to $0$ too. Denote the index of the last minor iteration as $l_k$ so that $x^{k+1}=x^{kl_k}$.
\begin{lemma}\label{dec:0}
The model reduction of LPBNC converges to 0. Specifically, \\
(i) if in iteration $k$ there is an infinite sequence of minor iterations then 
\be\label{modred:dec0}
\lim\limits_{l\rightarrow\infty}[m^k_l(x^k)-m^k_l(x^{kl})]=0;
\ee
(ii) if the sequence of major iteration points $\{x^k\}$ is infinite then 
\be\label{modred:serdec}
\lim\limits_{k\rightarrow\infty}[m^k_{l_k}(x^k)-m^k_{l_k}(x^{k+1})]=0.
\ee
\end{lemma}
\begin{proof}
(i) From Lemma \ref{reduction:lemma} we know $m^k_l(x^k)=f(x^k)$ for all $(k,l)$ and the sequence $\{f(x^k)-m^k_l(x^{kl})\}_{l=1}^{\infty}$ is nonnegative. 
From Lemma \ref{modnondecr} we see this sequence is also monotonic. 
Since the sequence is infinite, by Lemma \ref{MinorTermiOr} there is an infinite sequence of indices $0<l_1<l_2<\cdots$ such that $0\leq f(x^k)-m^k_{l_j}(x^{kl_j})\leq \bar{\eta}_2[f(x^k)-m^k_{l_{j-1}}(x^{kl_{j-1}})]\leq\cdots\leq \bar{\eta}_2^{j-1}[f(x^k)-m^k_{l_1}(x^{kl_{1}})]$ where $j$ can be infinitely large. 
Consequently, \eqref{modred:dec0} holds true.\\
(ii) From Remark \ref{f:dec}\eqref{rem2} we see the sequence $\{f(x^k)\}_{k=0}^{\infty}$ is monotonic. 
Under Assumption \ref{assump}, $f$ is bounded below. Hence $\lim\limits_{k\rightarrow\infty}[f(x^k)-f(x^{k+1})]=0$.\\ 
From the definition of $\rho^k_{l_k}$ and Lemma \ref{reduction:lemma}, we have $f(x^k)-f(x^{kl_k})\geq \eta_1 [f(x^k)-z^{kl_k}]=\left( m^k_{l_k}(x^k)-m^k_{l_k}(x^{kl_k})\right)$. Since both $\{f(x^k)-f(x^{k+1})\}$ and $\{m^k_{l_k}(x^k)-m^k_{l_k}(x^{k+1})\}$ are nonnegative we must have $\lim\limits_{k\rightarrow\infty}[m^k_{l_k}(x^k)-m^k_{l_k}(x^{kl_k})]=0$.
\end{proof}
Let $\bar{L}$ be the Lipschitz constant of $f$. 
We are now ready to prove the convergence theorem of LPBNC under Assumption 1.
\begin{theorem}\label{infini:minor}
Let Assumption 1 hold and $\epsilon_{\textrm tol}=0$. Suppose Algorithm \ref{LPBNC3} generates an infinite number of minor iterations after the $\bar{k}$-th major iteration. 
For every infinite subsequence $\mathcal{K}\subseteq\mathbb{N}$, if $B(\mathcal{K})\mathrel{\mathop:}=\{l\in\mathcal{K} \; | \;a^{\bar{k}}_l<a^{th}(x^{\bar{k}},I(\bar{k},l))\}$ is a finite set, then\\
(i) there exists $\{p_l\}_{l\in\mathcal{K}}$ such that $p_l\in P_{{a^{\bar{k}}_l}}(x^{\bar{k}})$ and $\norm{p_l-x^{\bar{k}}}\overset{\mathcal{K}}{\rightarrow}0$; 
\\
(ii) $0\in\partial f(x^{\bar{k}})$.
\end{theorem}
\begin{proof}
(i) Let $\mathcal{K}\subseteq \mathbb{N}$ be an infinite subsequence and $B(\mathcal{K})$ be a finite set, then there exists $N_1\in \mathcal{K}$ such that $a^{\bar k}_l\geq a^{th}(x^{\bar{k}},I(\bar{k},l))$ for all $l\geq N_1$ and $l\in\mathcal{K}$. 
By Lemma \ref{lm:34}, $a^{\bar{k}}_l$ satisfies the required conditions in Lemma \ref{local:mod} for all $l\geq N_1$ and $l\in\mathcal{K}$. 
Suppose for contradiction that for all sequences $\{p_l\}_{l\in \mathcal{K}}$ such that $p_l\in P_{{a^{\bar{k}}_l}}(x^{\bar{k}})$, there exist $\epsilon>0$ and $N_2\in\mathcal{K}$ such that $\norm{p_l-x^{\bar{k}}}\geq \epsilon$ for all $l\geq N_2$ and $l\in\mathcal{K}$. 
Then conclusion \eqref{mod:rd1} can be applied with $(\bar{x},\ x^*,\ a,\ \Delta,\ p,\ I)$ replaced by $(x^{\bar{k}},\ x^{\bar{k},l},\ a^{\bar{k}}_l,\ \Delta^{\bar{k}}_l,\ p_l,\ I(\bar{k},l))$ for all $l\geq N_3\mathrel{\mathop:}=\max\{N_1,N_2\}$ and $l\in\mathcal{K}$. 
For simplicity of notation we drop the superscript $\bar{k}$ and set $(x^l,\ a_l,\ \Delta_l,\ \rho_l)\mathrel{\mathop:}=(x^{\bar{k},l},\ a^{\bar{k}}_l,\ \Delta^{\bar{k}}_l,\ \rho^{\bar{k}}_l)$. We have 
\begin{equation}\label{aol3}
m(x^{\bar{k}})-m(x^l)\geq[f(x^{\bar{k}})-e_{{a_l}}(x^{\bar{k}})]
\min\{\frac{\Delta_l}{\norm{x^{\bar{k}}-p_l}},1\}\geq 0,\ \forall\ l\in\mathcal{K}_{\geq N_3}.
\end{equation}
From Lemma \ref{dec:0}(i) we have $m(x^{\bar{k}})-m(x^l)\rightarrow 0$ as $l\rightarrow \infty$. Consequently, \begin{equation}\label{contra:maj}
[f(x^{\bar{k}})-e_{{a_l}}(x^{\bar{k}})]
\min\{\frac{\Delta_l}{\norm{x^{\bar{k}}-p_l}},1\}\overset{\mathcal{K}}{\rightarrow}0.
\end{equation}
We first show that
\begin{equation}\label{fm:ea}
[f(x^{\bar{k}})-e_{{a_l}}(x^{\bar{k}})]\overset{\mathcal{K}}{\nrightarrow}0.
\end{equation}
Suppose for contradiction that $[f(x^{\bar{k}})-e_{{a_l}}(x^{\bar{k}})]\overset{\mathcal{K}}{\rightarrow}0.$ 
By Lemma \ref{lem:aBound}, $\{a_l\}_{l\in\mathcal{K}}$ is bounded, and hence there exist $a^*$ and $\bar{\mathcal{K}}\subseteq\mathcal{K}$ such that $a_l\overset{\bar{\mathcal{K}}}{\rightarrow}a^*$. 
As the mapping $e_{(\cdot)}(x^{\bar{k}})$ is continuous, we have $f(x^{\bar{k}})=e_{a^*}(x^{\bar{k}})$ and equivalently $x^{\bar{k}}\in P_{a^*}(x^{\bar{k}})$. 
From the outer semicontinuity of the mapping $P_{(\cdot)}(x^{\bar{k}})$, there exist $\epsilon'<\epsilon$ and $N_4\in\bar{\mathcal{K}}$ such that $\min\limits_{p\in P_{a_l}(x^{\bar{k}})}\norm{x^{\bar{k}}-p}\leq\epsilon'$, for all $l\in\bar{\mathcal{K}}_{\geq N_4}$. 
However, this cannot be true because $\epsilon'<\epsilon$, $\bar{\mathcal{K}}\subseteq\mathcal{K}$ and we have supposed that $\norm{p_l-x^{\bar{k}}}\geq \epsilon$ for all $N_2\leq l\in\mathcal{K}$, where $p_l$ can be an arbitrary element of $P_{a_l}(x^{\bar{k}})$.   
We have finished showing \eqref{fm:ea} which together with \eqref{contra:maj} yields 
\begin{equation}\label{donxmp}
\frac{\Delta_l}{\norm{x^{\bar{k}}-p_l}}\overset{\mathcal{K}}{\rightarrow}0.
\end{equation}
Next we show 
\begin{equation}\label{nxmpb}
\{\norm{x^{\bar{k}}-p_l}\}_{l\in\mathcal{K}}\text{ is bounded}.
\end{equation}
By Definition \ref{proximal}, 
$f(p_l)\leq e_{{a_l}}(x^{\bar{k}})\leq f(x^{\bar{k}})$, and therefore $p_l\in\text{lev}_{x^{\bar{k}}}f\subseteq\text{lev}_{x^0}f$.
We also have $x^{\bar{k}}\in\text{lev}_{x^0}f$ which is bounded. 
Hence $\{\norm{x^{\bar{k}}-p_l}\}_{l\in\mathcal{K}}$ is bounded above. We have supposed that $\norm{p_l-x^{\bar{k}}}\geq \epsilon$ for all $N_2\leq l\in\mathcal{K}$, hence $\{\norm{x^{\bar{k}}-p_l}\}_{l\in\mathcal{K}}$ is bounded below and \eqref{nxmpb} is true. From \eqref{donxmp} and \eqref{nxmpb} we have \begin{equation}\label{Dgt0}
\Delta_l\overset{\mathcal{K}}{\rightarrow}0.
\end{equation}
Then
 line \ref{decr:del} of Algorithm \ref{LPBNC3} must be executed infinite times, which implies that there exists an infinite subsequence $\mathcal{K^*}\subseteq \mathbb{N}$ such that $\rho_l<-\frac{1}{\min\{\Delta_l,1\}}$ for all $l\in\mathcal{K^*}$ and $\mathcal{K'}\mathrel{\mathop:}=\mathcal{K}\cap\mathcal{K^*}$ is infinite. 
 By the definition of $\rho_l$, the Lipschitz continuity of $f$, and the feasibility of $x^l$ to problem \eqref{subppc} we have \begin{equation}\label{deltaM} m(x^{\bar{k}})-m(x^l)<[f(x^{\bar{k}})-f(x^l)]\min\{\Delta_l,1\}\leq \bar{L}\norm{x^{\bar{k}}-x^l}\min\{\Delta_l,1\}\leq\bar{L}{\Delta_l}^2,\ \forall\ l\in\mathcal{K'}.
\end{equation}
By \eqref{Dgt0} there exists $N_5\in\mathcal{K}_{\geq N_2}$ such that $\Delta_l<\epsilon\leq \norm{p_l-x^{\bar{k}}}$ for all $l\in\mathcal{K}_{\geq N_5}$. 
This implies $\min\{\frac{\Delta_l}{\norm{x^{\bar{k}}-p_l}},1\}=\frac{\Delta_l}{\norm{x^{\bar{k}}-p_l}}$ and from \eqref{aol3} we have 
\begin{equation}\label{lcil2}
f(x^{\bar{k}})-e_{{a_l}}(x^{\bar{k}})\leq \frac{[m(x^{\bar{k}})-m(x^l)]\norm{x^{\bar{k}}-p_l}}{\Delta_l}\leq \bar{L}\norm{x^{\bar{k}}-p_l}\Delta_l,\ \forall\ l\in(\mathcal{K}_{\geq N_6}\cap\mathcal{K'}),
\end{equation}
where the last inequality follows from \eqref{deltaM} and $N_6\mathrel{\mathop:}=\max\{N_3,N_5\}$. From \eqref{nxmpb}, \eqref{Dgt0}, \eqref{lcil2} and the fact that $\mathcal{K'}\subseteq\mathcal{K}$ we have 
\begin{equation}\label{lcot42}
[f(x^{\bar{k}})-e_{{a_l}}(x^{\bar{k}})]\overset{\mathcal{K'}}{\rightarrow}0.
\end{equation}
In \eqref{fm:ea} and its poof the $\mathcal{K}$ can be replaced by any infinite subsequence $\tilde{\mathcal{K}}$ such that $\norm{p_l-x^{\bar{k}}}\geq \epsilon$ for all $l\in\tilde{\mathcal{K}}_{\geq N_2}$ with $p_l\in P_{a_l}(x^{\bar{k}})$. Consequently, \eqref{lcot42} cannot be true and we have found a contradiction.  
%
%
%
%
%
%
%
%
%
%
%
%
%
%
%

(ii) Finally, to see $0\in\partial f(x^{\bar{k}})$, note 
${a}_l(x^{\bar{k}}-p_l)\in\partial f(p_l)$ for all $p_l\in P_{{a_l}}(x^{\bar{k}})$ and $l\in\mathcal{K}$. By the outer semicontinuity of the proximal mapping and subdifferential, when $\norm{p_l-x^{\bar{k}}}\overset{\mathcal{K}}{\rightarrow}0$, as long as $\{a_l\}_{l\in\mathcal{K}}$ is bounded which is true from Lemma \ref{lem:aBound}, we have $0\in\partial f(x^{\bar{k}})$.
\end{proof}
We denote the minor iterations between the $k$-th major iteration and the $(k+1)$-th major iteration $M(k)\mathrel{\mathop:}=\{0,1,\ \cdots,\ l_k\}$ with $l_k=0$ if there is no minor iteration in between. We will need the following assumption.
\begin{asump}
If there exists a sequence of indices $\{j_k\}_{k\in \mathcal{K}}$ with $j_k\in M(k)$ such that $G\mathrel{\mathop:}=\{k\in\mathbb{N}|\rho^k_{j_k}<-\frac{1}{\min\{\Delta^k_{j_k},1\}}\}$ is an infinite set then $\{k\in G|a^k_{j_k}<a^{th}(x^k,I(k,j_k))\}$ is a finite set.
\end{asump}
\begin{theorem}\label{theo:final}
Let Assumption 1 and Assumption 2 hold and $\epsilon_{\textrm tol}=0$. 
Suppose Algorithm \ref{LPBNC3} generates an infinite number of major iterations. 
For every subsequence $\mathcal{K}\subseteq\mathbb{N}$ and $\hat{x}$ such that $x^k\overset{\mathcal{K}}{\rightarrow}\hat{x}$, there exists an associated sequence of indices $\{i_k\}_{k\in \mathcal{K}}$ with $i_k\in M(k)$, such that if $E(\mathcal{K})\mathrel{\mathop:}=\{k\in\mathcal{K}|a^k_{i_k}<a^{th}(x^k,I(k,i_k))\}$ is a finite set, then\\
(i) there exists $\{p^k\}_{k\in \mathcal{K}}$ with $p^k\in P_{{a^k_{i_k}}}(x^k)$ such that $\norm{x^k-p^k}\overset{\mathcal{K}}{\rightarrow}0$;
\\ (ii) $0\in\partial f(\hat{x})$. 
\end{theorem}
\begin{proof}
(i) Let $\mathcal{K}\subseteq \mathbb{N}$ be an infinite subsequence such that $x^k\overset{\mathcal{K}}{\rightarrow}\hat{x}$. 
Suppose for contradiction that for all sequences $\{p^k\}_{k\in \mathcal{K}}$ and $\{i_k\}_{k\in \mathcal{K}}$ with $i_k\in M(k)$, $p^k\in P_{{a^k_{i_k}}}(x^k)$ and $E(\mathcal{K})$ finite, there exist $\epsilon>0$ and $M_1\in\mathcal{K}$ such that $\norm{x^k-p^k}\geq \epsilon$ for all $k\in\mathcal{K}_{\geq M_1}$. We first show that  
\begin{equation}\label{1fm:ea}
[f(x^{\bar{k}})-e_{{a^k_{i_k}}}(x^{\bar{k}})]\overset{\mathcal{K}}{\nrightarrow}0.
\end{equation}
Suppose for contradiction that $[f(x^{\bar{k}})-e_{{a^k_{i_k}}}(x^{\bar{k}})]\overset{\mathcal{K}}{\rightarrow}0.$ 
By Lemma \ref{lem:aBound}, $\{a^k_{i_k}\}_{k\in\mathcal{K}}$ is bounded, and hence there exist $a^*$ and $\bar{\mathcal{K}}\subseteq\mathcal{K}$ such that $a^k_{i_k}\overset{\bar{\mathcal{K}}}{\rightarrow}a^*$. 
As both $f(\cdot)$ and $e_{(\cdot)}(\cdot)$ are continuous, we have $f(\hat{x})=e_{a^*}(\hat{x})$ and equivalently $\hat{x}\in P_{a^*}(\hat{x})$. 
From the outer semicontinuity of the mapping $P_{(\cdot)}(\cdot)$, there exist $\epsilon'<\epsilon$ and $M_2\in\bar{\mathcal{K}}$ such that $\min\limits_{p\in P_{a^k_{i_k}}(x^{{k}})}\norm{x^{{k}}-p}\leq\epsilon'$, for all $k\in\bar{\mathcal{K}}_{\geq M_2}$. 
However, this cannot be true because $\epsilon'<\epsilon$, $\bar{\mathcal{K}}\subseteq\mathcal{K}$ and we have supposed that $\norm{x^k-p^k}\geq \epsilon$ for all $k\in\mathcal{K}_{\geq M_1}$, where $p^k$ can be an arbitrary element of $P_{a^k_{i_k}}(x^k)$.
We have finished showing \eqref{1fm:ea}.

Take $i_k\equiv l_k$ for all $k\in\mathcal{K}$. Since $E(\mathcal{K})$ is a finite set, there exists $M_3\in \mathcal{K}$ such that $a^{k}_{l_k}\geq a^{th}(x^{{k}},I({k},l_k))$ for all $ k\in\mathcal{K}_{\geq M_3}$. 
By Lemma \ref{lm:34}, $a^{{k}}_{l_k}$ satisfies the required conditions in Lemma \ref{local:mod} for all $ k\in\mathcal{K}_{\geq M_3}$. 
Then conclusion \eqref{mod:rd1} can be applied with $(\bar{x},\ x^*,\ a,\ \Delta,\ p,\ I)$ replaced by $(x^{{k}},\ x^{{k},l_k},\ a^{{k}}_{l_k},\ \Delta^{{k}}_{l_k},\ p^k,\ I({k},l_k))$ for all $k\in \mathcal{K}_{\geq M_4}$ where $M_4\mathrel{\mathop:}=\max\{M_1,M_3\}$, i.e. 
\[
m(x^{{k}})-m(x^{k}_{l_k})\geq[f(x^{{k}})-e_{{a^k_{l_k}}}(x^{{k}})]
\min\{\frac{\Delta^k_{l_k}}{\norm{x^{{k}}-p^k}},1\}\geq 0,\ \forall\ k\in \mathcal{K}_{\geq M_4}.\]
From Lemma \ref{dec:0}(ii) we have $m(x^{{k}})-m(x^{k+1})\rightarrow 0$ as $k\rightarrow \infty$. Consequently, \begin{equation}\label{contra:maj2}
[f(x^k)-e_{{a^k_{l_k}}}(x^k)]
\min\{\frac{\Delta^k_{l_k}}{\norm{x^k-p^k}},1\}\overset{\mathcal{K}}{\rightarrow}0.
\end{equation}
From \eqref{1fm:ea} and \eqref{contra:maj2} we have 
\begin{equation}\label{1donxmp}
\frac{\Delta^k_{l_k}}{\norm{x^k-p^k}}\overset{\mathcal{K}}{\rightarrow}0.
\end{equation}
Next we show 
\begin{equation}\label{1nxmpb}
\{\norm{x^k-p^k}\}_{k\in\mathcal{K}}\text{ is bounded}.
\end{equation}
By Definition \ref{proximal}, 
$f(p^k)\leq e_{{a^k_{l_k}}}(x^k)\leq f(x^k)$, and therefore $p^k\in\text{lev}_{x^k}f\subseteq\text{lev}_{x^0}f$.
We also have $x^k\in\text{lev}_{x^0}f$ which is bounded. 
Hence $\{\norm{x^k-p^k}\}_{k\in\mathcal{K}}$ is bounded above.  
We have supposed that $\norm{x^k-p^k}\geq \epsilon$ for all $k\in\mathcal{K}_{\geq M_1}$, hence $\{\norm{x^k-p^k}\}_{k\in\mathcal{K}}$ is bounded below and \eqref{1nxmpb} is true. From \eqref{1donxmp} and \eqref{1nxmpb} we have \begin{equation}\label{dklkd0}
\Delta^k_{l_k}\overset{\mathcal{K}}{\rightarrow}0.
\end{equation}
Then
 line \ref{decr:del} of Algorithm \ref{LPBNC3} must be executed infinite times, which implies that there exists a subsequence $\mathcal{K^*}\subseteq \mathbb{N}$ such that $\mathcal{K'}\mathrel{\mathop:}=\mathcal{K}\cap\mathcal{K^*}$ is infinite and $\rho^k_{j_k}<-\frac{1}{\min\{\Delta^k_{j_k},1\}}$ for all $k\in\mathcal{K^*}$, with $j_k\in M(k)$.  
 By the definition of $\rho^k_{j_k}$, the Lipschitz continuity of $f$, and the feasibility of $x^{k,j_k}$ to problem \eqref{subppc} we have \begin{align}\label{deltaM2} m^k_{j_k}(x^k)-m^k_{j_k}(x^{k,j_k})&<[f(x^k)-f(x^{k,j_k})]\min\{\Delta^k_{j_k},1\}\notag\\
 &\leq \bar{L}\norm{x^k-x^{k,j_k}}\min\{\Delta^k_{j_k},1\}\notag\\
 &\leq\bar{L}{\Delta^k_{j_k}}^2,\ \forall\ k\in\mathcal{K'},
\end{align}
where $m^k_{j_k}(\cdot)\mathrel{\mathop:}=m(\cdot;x^k,a^k_{j_k},I(k,j_k))$.

We now show $\Delta^k_{j_k}\overset{\mathcal{K'}}{\rightarrow}0$. Suppose for contradiction that $\{\Delta^k_{j_k}\}_{k\in\mathcal{K'}}$ is bounded away from 0. 
From \eqref{dklkd0} and the fact that $\mathcal{K'}\subseteq\mathcal{K}$, there exists $M_5\in\mathcal{K'}$ such that 
\begin{equation}\label{delta:bo1}
\Delta^k_{l_k}<\Delta^{k+1}_{j_{k+1}} \text{ for all }k\in\mathcal{K'}_{\geq M_5}.
\end{equation}
From the update of trust region in Algorithm \ref{LPBNC3} we see that in minor iterations trust region radius is not increased and in major iterations trust region radius is increased under some conditions. Thus \eqref{delta:bo1} implies that 
\begin{equation}\label{delta:bo2}
\Delta^{k+1}_{j_{k+1}}\leq \Delta^{k+1}_0=\min\{\alpha_2\Delta^k_{l_k} ,\Delta_{\max}\}\leq \alpha_2\Delta^k_{l_k}\text{ for all }k\in\mathcal{K'}_{\geq M_5}.
\end{equation} 
From \eqref{dklkd0}, $\mathcal{K'}\subseteq\mathcal{K}$, \eqref{delta:bo1} and \eqref{delta:bo2}, we have
\begin{equation}\label{Dkjkgt0}
\Delta^k_{j_k}\overset{\mathcal{K'}}{\rightarrow}0.
\end{equation}
We have finished showing $\Delta^k_{j_k}\overset{\mathcal{K'}}{\rightarrow}0$ by reductio ad absurdum. 
Now for each $j_k$ there exists $p^k_{j_k}\in P_{a^k_{j_k}}(x^k)$. 
If $\norm{p^k_{j_k}-x^k}\overset{\mathcal{K'}}{\rightarrow} 0$ then we have found a sequence satisfying conclusion (i). 
Thus we suppose for contradiction that 
$\{\norm{p^k_{j_k}-x^k}\}_{k\in\mathcal{K'}}$ is bounded away from 0. 
Consequently, there exists $M_6\in\mathcal{K'}$ such that $\Delta^k_{j_k}<\epsilon\leq \norm{p^k_{j_k}-x^k}$ for all $k\in\mathcal{K'}_{\geq M_6}$. 
This implies $\min\{\frac{\Delta^k_{j_k}}{\norm{x^k-p^k_{j_k}}},1\}=\frac{\Delta^k_{j_k}}{\norm{x^k-p^k_{j_k}}}$ for all $k\in\mathcal{K'}_{\geq M_6}$. 
Under Assumption 2, $\{k\in G|a^k_{j_k}<a^{th}(x^k,I(k,j_k))\}$ is a finite set, and therefore there exists $M_7\in\mathcal{K'}$ such that Lemma \ref{local:mod} can be applied with $(\bar{x},\ x^*,\ a,\ \Delta,\ p,\ I)$ replaced by $(x^{{k}},\ x^{{k},j_k},\ a^{{k}}_{j_k},\ \Delta^{{k}}_{j_k},\ p^k_{j_k},\ I({k},j_k))$ for all $k\in \mathcal{K'}_{\geq M_7}$, i.e. 
\begin{equation}\label{deltaM3}
f(x^k)-e_{{a^k_{j_k}}}(x^k)\leq \frac{[m_{j_k}(x^k)-m_{j_k}(x^{k,j_k})]\norm{x^k-p^k_{j_k}}}{\Delta^k_{j_k}},\ \forall\ k\in\mathcal{K'}_{\geq M_8},
\end{equation}
where $M_8\mathrel{\mathop:}=\max\{M_6,M_7\}$. 
From \eqref{deltaM2} and \eqref{deltaM3} we have 
\begin{equation}\label{deltaM1}
f(x^k)-e_{{a^k_{j_k}}}(x^k)\leq \bar{L}\norm{x^k-p^k_{j_k}}\Delta^k_{j_k},\ \forall\ k\in\mathcal{K'}_{\geq M_8}.
\end{equation}
The $p^k$ in \eqref{1nxmpb} and its proof can be replaced by $p^k_{j_k}$ and thus $\{\norm{x^k-p^k_{j_k}}\}_{k\in\mathcal{K'}}$ is bounded. This together with \eqref{Dkjkgt0} and \eqref{deltaM1} yields $f(x^k)-e_{{a^k_{j_k}}}(x^k)\overset{\mathcal{K'}}{\rightarrow}0$. We can easily check that the $\mathcal{K}$ in \eqref{1fm:ea} and its proof can be replaced by $\mathcal{K'}$ with $i_k$ replaced by $j_k$ and $p^k$ replaced by $p^k_{j_k}$. Hence we have $f(x^k)-e_{{a^k_{j_k}}}(x^k)\overset{\mathcal{K'}}{\nrightarrow}0$. We have found a contradiction and conclusion (i) holds true. 

(ii) To see $0\in\partial f(\hat{x})$, note 
${a}^k_{i_k}(x^k-p^k)\in\partial f(p^k)$ for all $p^k\in P_{{a^k_{i_k}}}(x^k)$ and $k\in\mathcal{K}$. By the outer semicontinuity of the proximal mapping and subdifferential, when $x^k\overset{\mathcal{K}}{\rightarrow}\hat{x}$ and $\norm{p^k-x^k}\overset{\mathcal{K}}{\rightarrow}0$, as long as $\{a^k_{i_k}\}_{k\in\mathcal{K}}$ is bounded which is true from Lemma \ref{lem:aBound}, we have $0\in\partial f(\hat{x})$. 
\end{proof}
\section{Numerical experiments}
In this section we report some preliminary numerical results on implementations of LPBC and LPBNC. Here our goal is to provide a proof of principle only. For nonconvex examples we demonstrate that the conditions stated in the convergence theorems can be satisfied. The two algorithms were programmed in MATLAB R2012b in a computer with 3.40 GHz CPU and 8 GB RAM. We used the CPLEX connector (V12.5.1) for MATLAB to solve the linear subproblems. Specifically, problem \eqref{subpz2} was programmed and solved by the CPLEX Class API and problem \eqref{subppc} was solved by the toolbox function \textsf{cplexlp}. CPLEX automatically chooses from primal simplex, dual simplex and barrier optimizers to solve a given linear programming problem. In our implementations we found that no instances of the linear subproblems were solved by barrier optimizer. 
The implementation of our algorithms requires high accuracy of the solution of linear subproblems. Hence the CPLEX tolerances of optimality and feasibility are crucial to the performance of LPBC and LPBNC. As we see from the algorithms, both the stopping criterion and definition of $\rho$ are dependent on the model reduction, $f(\bar{x})-z^*$; if $z^*$ provided by CPLEX solver is slightly bigger than the actual optimal value of $\eqref{subppc}$, then the model reduction can be significantly inaccurate. In fact, we observed instances of negative model reduction when we use the default tolerances of optimality and feasibility in CPLEX. 
To prevent the occurrence of such cases, we set both the optimality and feasibility tolerances to $10^{-9}$, the least value in CPLEX.

In our experiments, the choice of the initial trust region radius $\Delta^0_0$ can significantly change the performance of our algorithm on some problems. 
In most trust region methods, choosing the initial trust region radius is an import issue, as stated in the monograph \cite[page 784]{R.2000} ``one very often has to resort to some heuristic to choose $\Delta_0$ on the basis of other initial information." Nonetheless, \cite{R.2000} suggested several strategies of initializing trust region and we adopted the following two choices, $\Delta^0_0=1$ and $\Delta^0_0=\frac{1}{10}\normt{s^0}$, where $s^0$ is an element of $\partial f(x^0)$. Settings for other trust region related parameters in both LPBC and LPBNC are $\eta_1=10^{-4},\ \eta_3=0.4,\ \ \alpha_1=0.25,\ \alpha_2=2$, and $\Delta_{\max}=1000$.
\subsection{Convex Examples}
Table \ref{tab:rfcp} presents the results for the tested convex problems listed in Table \ref{con:prob}, where   
 problems 1 to 15 are taken from section 3 of \cite{Luksan2000} and problems 16 to 20 are the problems 2.1 to 2.5 of \cite{Karmitsa2007}. For all convex and nonconvex problems in our tests we use the initial points provided in the associated references. 
\begin{table}[htbp]
  \centering
  \caption{Tested convex problems}
    \begin{tabular}{cccc}
    \toprule
    No.   & Problem & Dimension & Optimal Value \\
    \midrule
    1     & CB2   & 2     & 1.9522245 \\
    2     & CB3   & 2     & 2 \\
    3     & DEM   & 2     & -3 \\
    4     & QL    & 2     & 7.2 \\
    5     & LQ    & 2     & -1.4142136 \\
    6     & Mifflin1 & 2     & -1 \\
    7     & Wolfe & 2     & -8 \\
    8     & Rosen & 4     & -44 \\
    9     & Shor  & 5     & 22.600162 \\
    10    & Maxquad & 10    & -0.8414083 \\
    11    & Maxq  & 20    & 0 \\
    12    & Maxl  & 20    & 0 \\
    13    & Goffin & 50    & 0 \\
    14    & MXHILB & 50    & 0 \\
    15    & L1HILB & 50    & 0 \\
    16    & Generalization of MAXQ & 100   & 0 \\
    17    & Generalization of MXHILB & 100   & 0 \\
    18    & Chained LQ & 100   & -$99\sqrt{2}$ \\
    19    & Chained CB3 I & 100   & 198 \\
    20    & Chained CB3 II & 100   & 198 \\
    \bottomrule
    \end{tabular}%
  \label{con:prob}%
\end{table}%
The following abbreviations are used in Table \ref{tab:rfcp}.\\
\begin{tabular}{ll}
$f_{val}$ & minimal function value returned by the algorithm,\\
nf & number of function evaluations used by the algorithm,\\
$k$ & number of major iterations,\\
L & number of minor iterations,\\
time & elapsed CPU time,\\
t-CPX & sum of CPU time by CPLEX solver,\\
$\Delta$ & final value of trust region radius,\\
sh & number of times that trust region radius is decreased,\\
pr & number of times that primal simplex method is chosen by CPLEX,\\
dual & number of times that dual simplex method is chosen by CPLEX. 

\end{tabular}
\begin{table}[htbp]
  \centering
  \caption{Results for convex problems}
    \begin{tabular}{ccccccccccc}
    \toprule
    No.   & $f_{val}$ & nf & $k$   & L     & time  & t-CPX & $\Delta$ & sh & pr & dual \\
    \midrule
    1     & 1.952225451 & 16    & 10    & 5     & 0.1428 & 0     & 0.341842126 & 0     & 0     & 16 \\
    2     & 2     & 3     & 1     & 1     & 0.0977 & 0     & 0.5   & 1     & 0     & 3 \\
    3     & -3    & 8     & 6     & 1     & 0.1097 & 0     & 2.039607805 & 0     & 2     & 6 \\
    4     & 7.20000069 & 16    & 10    & 5     & 0.1322 & 0     & 0.141421356 & 2     & 0     & 16 \\
    5     & -1.41421274 & 18    & 14    & 3     & 0.1305 & 0     & 1.13137085 & 0     & 3     & 15 \\
    6     & -0.999999683 & 28    & 12    & 15    & 0.1562 & 0     & 0.03125 & 6     & 5     & 23 \\
    7     & -8    & 5     & 3     & 1     & 0.1021 & 0     & 2     & 1     & 0     & 5 \\
    8     & -43.99998585 & 54    & 24    & 29    & 0.2361 & 0     & 0.25  & 4     & 1     & 53 \\
    9     & 22.60018019 & 55    & 22    & 32    & 0.2356 & 0     & 0.25  & 1     & 0     & 55 \\
    10    & -0.841407474 & 220   & 34    & 185   & 0.7292 & 0     & 0.25  & 1     & 0     & 220 \\
    11    & 4.06847E-07 & 249   & 116   & 132   & 0.8042 & 0     & 2     & 3     & 0     & 249 \\
    12    & 0     & 36    & 16    & 19    & 0.1735 & 0     & 3.2   & 2     & 7     & 29 \\
    13    & 0     & 51    & 48    & 2     & 0.2316 & 0.016 & 4     & 0     & 1     & 50 \\
    14    & 2.35525E-07 & 15    & 10    & 4     & 0.1290 & 0     & 8.158764414 & 2     & 3     & 12 \\
    15    & 2.08721E-06 & 27    & 15    & 11    & 0.2028 & 0     & 0.13964447 & 6     & 1     & 26 \\
    16    & 4.21692E-07 & 1361  & 583   & 777   & 4.5585 & 0.079 & 40    & 2     & 1     & 1360 \\
    17    & 9.97664E-07 & 25    & 15    & 9     & 0.1571 & 0     & 2.045863824 & 5     & 3     & 22 \\
    18    & -140.0070287 & 1185  & 91    & 1093  & 8.4852 & 3.661 & 0.124058958 & 3     & 0     & 1185 \\
    19    & 198.000171 & 1437  & 132   & 1304  & 10.4948 & 5.109 & 0.13978045 & 5     & 2     & 1435 \\
    20    & 198.0000905 & 35612 & 199   & 35412 & 285.4227 & 150.166 & 0.13978045 & 5     & 2     & 35610 \\
    \bottomrule
    \end{tabular}%
  \label{tab:rfcp}%
\end{table}%

In LPBC, we set $\epsilon_{\mathrm{tol}}=10^{-6}$ and $T=30$. For problems 1 - 14 we initialize trust region radius by 1; for problems 15 - 20 we initialize trust region radius by $\frac{1}{10}\normt{s^0}$. 
From Table \ref{tab:rfcp} we see that LPBC returned optimal values to accuracy $10^{-6}$ for 14 problems, $10^{-4} $ for 4 problems and $2\times10^{-4}$ for 2 problems. For all problems except the last three, CPLEX consumed negligible time to solve all the linear subproblems. The possibility that trust region is shrunk (the value sh/L) is very small for the majority of the tested problems. We also see that most of the linear subproblems were solved by dual simplex method.   

\subsection{Nonconvex Examples} 
The tested nonconvex problems are listed in Table \ref{tl:1} where problems 1 to 7 are taken from section 3 of \cite{Luksan2000} and problems 8 to 12 are the problems 2.6 to 2.10 in \cite{Karmitsa2007}. 

\begin{table}[htbp]
  \centering
  \caption{Tested nonconvex problems
   }
    \begin{tabular}{cccc}
    \toprule
    No.   & Problem & Dimension & Optimal Value \\
    \midrule
    1     & Crescent & 2     & 0 \\
    2     & Mifflin2 & 2     & -1 \\
    3     & Colville 1 & 5     & -32.348679 \\
    4     & HS78  & 5     & -2.9197004 \\
    5     & El-Attar & 6     & 0.5598131 \\
    6     & Gill  & 10    & 9.7857721 \\
    7     & Steiner 2 & 12    & 16.703838 \\
    8     & Active Faces & 50    & 0 \\
    9     & Brown 2 & 50    & 0 \\
    10    & Chained Mifflin2 & 50    & -34.795 \\
    11    & Chained Crescent I & 50    & 0 \\
    12    & Chained Crescent II & 50    & 0 \\
\bottomrule
    \end{tabular}%
  \label{tl:1}
\end{table}
In LPBNC, we set $\gamma=2$ and $\epsilon_{\mathrm{tol}}=10^{-5}$. Apart from the initial trust region radius $\Delta_0^0$, another parameter that can cause  dramatic changes in the performance of LPBNC is the backtrack parameter $\beta$. 
Results with two settings of these parameters are listed in Table \ref{tab:4} and \ref{tab:5} below. We use the same abbreviations as in Table \ref{tab:rfcp}. 
Additionally, we list the difference of $f_{val}$ and optimal value (error), the number of function evaluations in the backtrack process divided by the total number of function evaluations (pb), the number of subgradient evaluations (se), the final value for $a$, the final value for $a^{\min}$ and the number of times that $a$ is updated (au). Subgradient evaluations do not happen in the backtrack process and hence the total number of subgradient evaluations is not equal to that of function evaluations. 
We found that all instances of LP subproblems were solved by dual simplex method in our tests for nonconvex problems.
\begin{table}[htbp]
  \centering
  \caption{Results for nonconvex problems with $\Delta^0_0=1,\ \beta=0.7$}
    \begin{tabular}{rrrrrrrrrrrrrr}
    \toprule
    No.   & error & nf & pb    & se    & $k$   & L     & time  & t-CPX & $\Delta$ & sh & a     & $a^{\min}$ & au \\
    \midrule
    1     & 0.000257 & 46    & 30.43 & 32    & 19    & 12    & 1.186 & 0     & 0.0625 & 3     & 0.90002 & 0.89339 & 3 \\
    2     & 0     & 20    & 10.00 & 18    & 4     & 13    & 0.260 & 0     & 0.03125 & 4     & 0     & 0     & 0 \\
    3     & 6.28E-07 & 46    & 19.57 & 37    & 14    & 22    & 0.561 & 0     & 0.125 & 3     & 0     & 0     & 0 \\
    4     & 6.5E-05 & 81    & 2.47  & 79    & 30    & 48    & 0.940 & 0     & 0.125 & 4     & 50.37528 & 48.16938 & 5 \\
    5     & 9.25E-07 & 99    & 9.09  & 90    & 35    & 54    & 1.214 & 0.031 & 0.125 & 8     & 0.40795 & 0.39540 & 9 \\
    6     & 0.0002196 & 205   & 1.95  & 201   & 30    & 170   & 4.602 & 0.076 & 0.015625 & 4     & 6.769637 & 6.769637 & 5 \\
    7     & 0.0001409 & 123   & 16.26 & 103   & 36    & 66    & 1.116 & 0.03  & 0.03125 & 3     & 0     & 0     & 0 \\
    8     & 0     & 2     & 0.00  & 2     & 1     & 0     & 0.278 & 0     & 2     & 0     & 0.118057 & 0.118057 & 1 \\
    9     & 0     & 2     & 0.00  & 2     & 1     & 0     & 0.121 & 0     & 2     & 0     & 0     & 0     & 0 \\
    10    & 4.05E-05 & 11549 & 69.81 & 3487  & 85    & 3401  & 345.646 & 16.931 & 0.003906 & 5     & 0     & 0     & 0 \\
    11    & 7.29E-06 & 5743  & 38.22 & 3548  & 189   & 3358  & 126.861 & 9.896 & 0.003906 & 5     & 0     & 0     & 0 \\
    12    & 1.01E-05 & 1637  & 54.31 & 748   & 72    & 675   & 15.801 & 1.199 & 0.000977 & 9     & 0     & 0     & 0 \\
    \bottomrule
    \end{tabular}%
  \label{tab:4}%
\end{table}%
We see in Table \ref{tab:4} problem 8 and 9 have error 0 with two function evaluations. These two instances are accidental as we found that the optimal solution of the first linear subproblem \eqref{subppc} is already the minimizer of the  objective function for these two problems given that we set $\Delta^0_0=1$ and we use the default initial points.
\begin{table}[htbp]
  \centering
  \caption{Results for nonconvex problems with $\Delta^0_0=\frac{1}{10}\normt{s^0},\ \beta=0.8$ }
    \begin{tabular}{cccccccccccccc}
    \toprule
    No.   & error & nf & pb    & se    & $k$   & L     & time  & t-CPX & $\Delta$ & sh    & a     & $a^{\min}$ & au \\
    \midrule
    1     & 0.0002435 & 40    & 37.50 & 25    & 19    & 5     & 0.311 & 0     & 0.10607 & 3     & 0.375 & 0.363 & 3 \\
    2     & 1.095E-05 & 20    & 5.00  & 19    & 6     & 12    & 0.247 & 0.016 & 0.03542 & 4     & 0     & 0     & 0 \\
    3     & 0.0001802 & 49    & 28.57 & 35    & 15    & 19    & 0.390 & 0     & 0.02140 & 6     & 0     & 0     & 0 \\
    4     & -6.636E+19 & 15    & 0.00  & 15    & 14    & 0     & 0.215 & 0     & 1000  & 0     & 3.78E+12 & 2.81E+12 & 13 \\
    5     & 1.587E-05 & 437   & 48.97 & 223   & 105   & 117   & 2.738 & 0.015 & 0.13473 & 28    & 0.196 & 0.163 & 49 \\
    6     & 0.0002527 & 458   & 26.64 & 336   & 75    & 260   & 4.295 & 0.11  & 0.04621 & 6     & 7.995 & 0     & 15 \\
    7     & 0.0001122 & 118   & 14.41 & 101   & 33    & 67    & 1.071 & 0.03  & 0.06833 & 3     & 0     & 0     & 0 \\
    8     & 0.0003365 & 1028  & 24.42 & 777   & 73    & 703   & 18.732 & 1.076 & 0.00693 & 16    & 0.027472 & 0.027465 & 46 \\
    9     & 3.334E-06 & 657   & 34.70 & 429   & 36    & 392   & 9.414 & 0.712 & 0.00136 & 12    & 0.005 & 0.003 & 12 \\
    10    & 2.918E-05 & 22741 & 81.43 & 4222  & 88    & 4133  & 369.968 & 18.958 & 0.00544 & 7     & 0     & 0     & 0 \\
    11    & 6.576E-06 & 9530  & 51.29 & 4642  & 180   & 4461  & 201.156 & 14.651 & 0.00238 & 7     & 2.485 & 0     & 4 \\
    12    & 8.237E-06 & 1801  & 64.02 & 648   & 73    & 574   & 21.416 & 1.587 & 0.00059 & 9     & 0     & 0     & 0 \\
    \bottomrule
    \end{tabular}%
  \label{tab:5}%
\end{table}%
We note that a huge negative error occurred in problem 4 in Table \ref{tab:5} because the problem HS78 is actually unbounded below and its optimal value in Table \ref{tl:1} is a local minimum. From Table \ref{tab:4} and \ref{tab:5} we see that the proportion of function evaluations in backtrack process can be very high. Problem 10, which has the biggest number of function evaluations also yields the highest possibility to backtrack. For all the problems, solving LP subproblems took a small amount of time. 
objective function. This special case appeared in all the tested Fuse errier Polynomials and Active Faces. We ran LPBNC on these problems with all choices of dimension $n\in\{2,10,10^2,10^3,10^4,10^5,10^6\}$. The starting point is $\vec{x^0}=[1,1,\dots,1]$, initial trust region radius $\Delta^0_0=1$. All theses instances terminated with stopping criterion reached, $f_{val}=0$, and $nf=3$. 
\section{Concluding remarks}
We present a version of bundle method with the unusual feature of using only an LP solver as its algorithmic engine.
We study the properties of the linear model and expressed its model reduction. 
The optimal solution of our linear subproblem is not unique in contrast to the case of quadratic subproblem. 
However, no significant information is lost in order to ensure convergence of the algorithm. We use a local convexificaton with the deletion of some cutting planes at the end of a major iteration. Preliminary numerical experiments show that the algorithm is reliable and efficient for solving convex problems. For functions that are locally Lipschitz continuous and prox-regular we show that upon successful convexification the algorithm can converge to a fixed point of the proximal mapping. Numerical results of nonconvex problems suggest that with insignificant time spent on solving LP subproblems, a big portion of function evaluations can be consumed in the backtrack process. Improvements such as incorporation of a line search can be further studied in the future in order to increase efficiency to enable the solution of large-scale problems. 
\bibliographystyle{spmpsci}
\bibliography{bibs}
\appendix
\section{Appendix}
\subsection{Proof for Theorem \ref{conv:th}}
In order to prove Theorem \ref{conv:th} we need the following preliminaries. Let $X$ denote a space with $\Vert \cdot \Vert $ its norm. For every $A\subseteq X$, $\alpha >0$ and $x^{\ast }\in X^{\ast }$ we
define 
\begin{equation*}
SL(x^{\ast },A,\alpha )=\left\{ x\in A\mid \langle x,x^{\ast }\rangle
>S(A,x^{\ast })-\alpha \right\} \mbox{\rm ,}
\end{equation*}%
where $S(A,x^{\ast }):=\sup \left\{ \langle x^{\ast },x\rangle \mid x\in
A\right\} $. We will denote a linear function on $X$ by $x^{\ast }$ or $\langle x^{\ast
},\cdot \rangle $ and their action an element $x$ by $x^{\ast }(x)$ or $%
\langle x^{\ast },x\rangle $. Denote the set of all strongly exposed points
of a set $C$ by $\exp \,C$ and the set of all extremal points by $\func{ext}%
\,C$. We have for any closed, proper convex function $h$ 
\begin{equation*}
\func{epi}h=\overline{\func{co}}\left( \func{ext}\func{epi}h\right) +\mathbf{%
R}_{+}=\overline{\func{co}}\left( \exp \func{epi}h\right) +\mathbf{R}_{+}.
\end{equation*} Suppose $(x,h(x))\in \func{ext}\,\func{epi}\,h$ then $h(x)=g(x)$.%
  In this section we use the following definition of proximal mapping 
$
P_{\lambda}\left( f\right) (x)\mathrel{\mathop:}=\arg\min\left\{ f(\cdot)+\frac{1}{2\lambda }%
\Vert x-\cdot\Vert^{2}\right\}
$ and $\overline{r}(f)$ is the associated prox-threshold. It is known that if $f:\mathbf{R}^{n}\rightarrow \mathbf{R}_{+\infty }$
is quadratically minorized by $\alpha -r\Vert \cdot \Vert ^{2}$ and $x\in 
\func{dom}f$ with $0<\lambda <\overline{r}(f,x)$ then for each $\beta
>f\left( x\right) $ we have $\ \exists \delta >0$ such that $\ $ 
\begin{equation}\label{lem:bound}
\Vert z-y\Vert \leq 2\lambda \frac{\left( \beta -\alpha +\frac{r}{2}\right)
\left( 1-2\delta \right) }{\left( 1-2\left( \delta +r\lambda \right) \right) 
}:=M_{\lambda } \ \forall\  y\in P_{\lambda }(f)(z)\ \forall\ z\in B_{\delta }(x).
\end{equation}%

\begin{lemma}
\label{lem:4}Suppose $h:D\rightarrow \mathbf{R}$ is an lower bounded,
real--valued lower semi--continuous convex function defined on a bounded
closed convex set $D\subseteq \mathbf{R}^{n},$ with interior. Then there
exists a $(y^{\ast },-1)$ strongly exposing $\text{epi }\,h$ at $(\bar{y},h(%
\bar{y}))$ if and only if $h-y^{\ast }$ achieves a strict minimum on $D$ at $%
\bar{y}$. That is, $\arg \min \left[ h-y^{\ast }\right] =\left\{ y\right\} $
or $0\in \partial \left[ h-y^{\ast }\right] \left( y\right) $ has a unique
solution $y$.
\end{lemma}

\begin{proof}
Indeed if $h-y^{\ast }$ achieves a strict minimum $\bar{\beta}:=\min_{y\in
D}\left\{ h(y)-\langle y^{\ast },y\rangle \right\} $ then we must have the
lower level sets 
\begin{equation*}
L(\beta ):=\left\{ y\in D\mid \beta >h(y)-\langle y^{\ast },y\rangle
\right\} 
\end{equation*}%
satisfying $\cap _{\beta <\bar{\beta}}L(\beta )=\{\bar{y}\}$. Otherwise
there would exists $y_{m}\in L(\beta _{m})$ for $\beta _{m}\downarrow \bar{%
\beta}$ converging to $y\neq \bar{y}$ with $y\in \cap _{\beta <\bar{\beta}%
}L(\beta )$ and by the lower continuity of $h$ we have 
\begin{equation*}
\liminf_{m}\beta _{m}=\bar{\beta}\geq \liminf_{m}\left( h(y_{m})-\langle
y^{\ast },y_{m}\rangle \right) \geq h(y)-\langle y^{\ast },y\rangle 
\end{equation*}%
implying $y\in \arg \min \left( h-y^{\ast }\right) .$ This contradicts the
uniqueness of the minimizer. Thus the slices 
\begin{equation*}
L(\bar{\beta}-\delta )=\left\{ y\in D\mid h(\bar{y})-\langle y^{\ast },\bar{y%
}\rangle +\delta >h(y)-\langle y^{\ast },y\rangle \right\} \rightarrow
\left\{ \bar{y}\right\} 
\end{equation*}%
as $\delta \downarrow 0$. 
Observe that $S(\text{epi }\,h,(-y^{\ast },1))=\langle
(\bar{y},h(\bar{y})),(-y^{\ast },1)\rangle $ and so 
\begin{equation*}
L(\bar{\beta}+\delta )=\left\{ y\in D\mid \langle (y,h(y)),(-y^{\ast
},1)\rangle >S(\text{epi }\,h,(-y^{\ast },1))-\delta \right\} .
\end{equation*}%
Finally note that on $\text{epi }\,h$ we have 
\begin{eqnarray}
SL((-y^{\ast },1),\text{epi }\,h,\delta ) &=&\{(y,\alpha )\in \text{epi }%
\,h\mid \langle (y,\alpha ),(-y^{\ast },1)\rangle >S(\text{epi }\,h,(-y^{\ast
},1))-\delta \}  \notag \\
&&\qquad \qquad \qquad \qquad   \label{eqn:6}
\end{eqnarray}%
and $\text{diam}\,SL((-y^{\ast },1),\text{epi }\,h,\delta )\leq \,\text{diam}%
L(\bar{\beta}+\delta )\times \delta \rightarrow 0$ as $\delta \downarrow 0$.

If $(y^{\ast },-1)$ strongly exposing $\text{epi }\,h$ at $(\bar{y},h(\bar{y}%
))$ then (\ref{eqn:6}) defines a slice whose diameter tends to zero. It then
holds that the projection of this onto $D$ also has a diameter tending to
zero from which it follows that $\text{diam}\,L(\bar{\beta}+\delta
)\rightarrow 0$ as $\delta \downarrow 0$. Clearly we have then $\cap _{\beta
<\bar{\beta}}L(\beta )=\{\bar{y}\}$ and so $h-y^{\ast }$ achieves a strict
mimimum $\bar{\beta}$ at $\bar{y}$.
\end{proof}
\begin{lemma}
\label{lem:5} Suppose $h:D\rightarrow \mathbf{R}$ is an lower bounded,
real--valued lower semi--continuous, convex function defined on a bounded
closed convex set $D\subseteq \mathbf{R}^{n},$ with interior. Suppose $%
(x,h(x))$ is an extremal point of $\func{epi }\,h$ and suppose $%
\{(x_{m},\alpha _{m})\}_{m=0}^{\infty }\subseteq \func{epi }\,h$ with $%
x_{m}=\sum_{i=0}^{n}\lambda _{i}^{m}x_{i}^{m}\rightarrow x$ and $\alpha
_{m}=\sum_{i=0}^{n}\lambda _{i}^{m}h(x_{i}^{m})\rightarrow h(x)$. Then
either $\lambda _{i}^{m}\rightarrow 0$ or $x_{i}^{m}\rightarrow $ $x$ and
hence for some $i$ we have $x_{i}^{m}\rightarrow $ $x$.
\end{lemma}
\begin{proof}
We now claim that either $\lambda _{i}^{m}\rightarrow 0$ or $%
x_{i}^{m}\rightarrow $ $x$ for otherwise by a compactness argument (on $%
[0,1]\times D$) we could extract a convergent pair of sub-sequences such
that after renumbering we would have $\lambda _{i}^{m}\rightarrow \lambda
\neq 0$ and $x_{i}^{m}\rightarrow x^{\prime }\neq x$ (note that $\lambda \in
\lbrack 0,1]$ and $x^{\prime }\in D$). Then as $\lambda
_{i}^{m}x_{i}^{m}\rightarrow \lambda x^{\prime }$ converges as does $%
\sum_{i=0}^{n}\lambda _{i}^{m}x_{i}^{m}\rightarrow x$ we have convergence of 
$\sum_{j\neq i}\left( \frac{\lambda _{j}^{m}}{1-\lambda _{i}^{m}}\right)
x_{j}^{m}$ to $x^{\prime \prime }\in D$ via 
\begin{equation*}
\lambda _{i}^{m}x_{i}^{m}+\left( 1-\lambda _{i}^{m}\right) \sum_{j\neq
i}\left( \frac{\lambda _{j}^{m}}{1-\lambda _{i}^{m}}\right)
x_{j}^{m}\rightarrow \lambda x^{\prime }+\left( 1-\lambda \right) x^{\prime
\prime }=x.
\end{equation*}%
It is not possible for $\lambda =1$ because this implies $x^{\prime }=x$.
Next note that 
\begin{equation}
\lambda _{i}^{m}h(x_{i}^{m})+\left( 1-\lambda _{i}^{m}\right) \sum_{j\neq
i}\left( \frac{\lambda _{j}^{m}}{1-\lambda _{i}^{m}}\right)
h(x_{j}^{m})\rightarrow h(x)  \label{eqn:4}
\end{equation}%
and by lower semi--continuity of $h$ we have $\liminf_{m}h(x_{i}^{m})\geq
h(x^{\prime })$. If we take a subsequence along which the limit infimum $%
\liminf_{m}h(x_{i}^{m})$ is achieved then (\ref{eqn:4}) implies $\left\{
\lim_{m_{k}}\sum_{j\neq i}\left( \frac{\lambda _{j}^{m_{k}}}{1-\lambda
_{i}^{m_{k}}}\right) h(x_{j}^{m_{k}})\right\} $ converges along this
subsequence as well. Let 
\begin{equation*}
\liminf_{m_{k}}\sum_{j\neq i}\left( \frac{\lambda _{j}^{m_{k}}}{1-\lambda
_{i}^{m_{k}}}\right) (x_{j}^{m_{k}},h(x_{j}^{m_{k}}))=(x^{\prime \prime
},\alpha ^{\prime \prime })\in \func{epi }\,h,
\end{equation*}%
with $h(x^{\prime \prime })\leq \alpha ^{\prime \prime }$ by the definition
of $\func{epi }\,h$. Then from (\ref{eqn:4}) and the lower semi--continuity
of $h$ again it follows that 
\begin{equation*}
\lambda h(x^{\prime })+\left( 1-\lambda \right) h(x^{\prime \prime })\leq
h(x)
\end{equation*}%
Thus there exists $(x^{\prime },\alpha ^{\prime })$, $(x^{\prime \prime
},\alpha ^{\prime \prime })\in \func{epi }\,h$ such that $\lambda (x^{\prime
},\alpha ^{\prime })+\left( 1-\lambda \right) (x^{\prime \prime },\alpha
^{\prime \prime })=(x,h(x))$ with $0<\lambda <1$ contradicting the
assumption that $(x,h(x))$ is an extremal point of $\func{epi }\,h$.

As either $\lambda _{i}^{m}\rightarrow 0$ or $x_{i}^{m}\rightarrow $ $x$ it
follows that for some $i$ we have $x_{i}^{m}\rightarrow $ $x$ since $%
\sum_{i=0}^{n}\lambda _{i}^{m}=1$ precludes all $\lambda _{i}^{m}$ from
tending to zero.
\end{proof}
\begin{proposition}
\label{prop:noconv}Suppose $h:D\rightarrow \mathbf{R}$ is an lower bounded,
real--valued lower semi--continuous, convex function defined on a bounded
closed convex set $D\subseteq \mathbf{R}^{n},$ with interior. In addition
suppose the strongly exposed points on $h$ are dense on the boundary of $%
\text{epi }h|_{C}$ where $C\subseteq D$ and $C$ is closed with $\text{int}%
C\neq \emptyset $. Define $g$ via $\text{epi }g:=\overline{\exp \text{epi }%
h|_{C}}$ then 
\begin{equation*}
h\left( x\right) =g\left( x\right) \quad \text{for }x\in C\text{. }
\end{equation*}
\end{proposition}

\begin{proof}
On $C$ take a $\left( x,h\left( x\right) \right) \in \exp \text{epi }h$ and
any supporting hyperplane generated by $\left( -y^{\ast },1\right) $. By
definition of exposedness we have 
\begin{equation*}
\text{epi }h\cap \left\{ \left( y,\alpha \right) \mid \langle (-y^{\ast
},1),(x,h(x))\rangle \geq \langle (-y^{\ast },1),(y,\alpha )\rangle \right\}
=\left\{ \left( x,h\left( x\right) \right) \right\} .
\end{equation*}%
Now 
\begin{eqnarray*}
\text{epi }h|_{C}\  &=&\text{co }\exp \text{epi }h|_{C}=\cap _{\substack{ %
\left( x,h\left( x\right) \right) \in \exp \text{epi }h \\ x\in C}}\left\{
\left( y,\alpha \right) \mid \langle (-y^{\ast },1),(x,h(x))\rangle \leq
\langle (-y^{\ast },1),(y,\alpha )\rangle \right\} \cap \left[ C\times 
\mathbf{R}\right]  \\
&\subseteq &\text{epi }h.
\end{eqnarray*}%
Now suppose there exists $\left( y,\alpha \right) \in \text{co }\exp \text{epi%
}h|_{C}\cap \left( \text{epi }g\right) ^{c}$. Using the density of $\exp 
\text{epi }h|_{C}$ we have 
\begin{equation*}
g\left( y\right) =\liminf_{\substack{ \left( x,h\left( x\right) \right) \in
\exp \text{epi }h \\ x\in C\text{, }x\rightarrow y}}h\left( x\right) .
\end{equation*}%
As $\left( x,h\left( x\right) \right) \in \exp \text{epi }h\subseteq \text{epi%
}h$ with $x\in C$ and $\left( y,\alpha \right) \in \text{co }\exp \text{epi }%
h|_{C}\ $implies\\ $\langle (-y^{\ast },1),(x,h(x))\rangle \leq \langle
(-y^{\ast },1),(y,\alpha )\rangle $ we have 
\begin{equation*}
\liminf_{\substack{ \left( x,h\left( x\right) \right) \in \exp \text{epi }h
\\ x\in C\text{, }x\rightarrow y}}\langle (-y^{\ast },1),(x,h(x))\rangle
=\langle (-y^{\ast },1),(y,g(y))\rangle \leq \langle (-y^{\ast
},1),(y,\alpha )\rangle .
\end{equation*}%
Thus $g(y)\leq \alpha $ or $\left( y,\alpha \right) \in \text{epi }g,$ a
contradiction. Hence $\text{co }\exp \text{epi }h|_{C}\cap \left[ C\times 
\mathbf{R}\right] \subseteq \text{epi }g$ giving the equality $\text{epi }%
h|_{C}\ =\text{co }\exp \text{epi }h|_{C}\cap \left[ C\times \mathbf{R}\right]
=\text{epi }g.$
\end{proof}

\begin{corollary}
\label{cor:convex}Suppose $g:D\rightarrow \mathbf{R}$ is an lower bounded,
real--valued  continuous function defined on a bounded closed set $%
D\subseteq \mathbf{R}^{n},$ with interior. If the strongly exposed points of 
$\text{epi }g$ are dense in the boundary of $\text{epi }g$ then $\,g\left(
x\right) =\text{co }\,\,g\left( x\right) $ for all $x\in D$ and hence $g$ is
the restriction to $D$ of a convex function. In particular this is true when 
$0\in \partial \left[ g-y^{\ast }\right] \left( y\right) $ has a unique
solution for all $y^{\ast }\in B_{1}\left( 0\right) $ such that $\left(
-y^{\ast },1\right) $ supports $\text{epi }g$.
\end{corollary}

\begin{proof}
Let $h:=\text{co }\text{epi }g=\text{epi }\text{co }g$ which has a convex domain 
$\text{co }D$ with $\text{int}\text{co }D\supseteq \text{int}D\neq \emptyset .$
Clearly a strongly exposed point of $h$ must be a strongly exposed point of $%
g$. We claim the converse is true. Let $\left( x,g\left( x\right) \right) $
be strongly exposed in $\text{epi }g$. Suppose $\left( x,h\left( x\right)
\right) $ is not strongly exposed in $\text{epi }h$ then there exists $%
(-y^{\ast },1$ $)$ supporting $\text{epi }\,h$ at $(x,h(x))$ such that for
some $\delta >0$ and all $\varepsilon >0$ there exists $(y^{\varepsilon
},h(y^{\varepsilon }))\in \text{epi }\,h$ with $\left\Vert (y^{\varepsilon
},h(y^{\varepsilon }))-\left( x,h\left( x\right) \right) \right\Vert \geq
\delta >0$ and 
\begin{equation*}
\langle (-y^{\ast },1),(x,h(x))\rangle +\varepsilon \geq \langle (-y^{\ast
},1),(y,h\left( y^{\varepsilon }\right) )\rangle .
\end{equation*}%
As $g\left( x\right) \geq h\left( x\right) $ and $h\left( y^{\varepsilon
}\right) =\sum_{i=0}^{n}\lambda _{i}^{\varepsilon }g\left(
x_{i}^{\varepsilon }\right) $ for some $\sum_{i=0}^{n}\lambda
_{i}^{\varepsilon }=1$, $\lambda _{i}^{\varepsilon }\geq 0$ and $%
y^{\varepsilon }=\sum_{i=0}^{n}\lambda _{i}^{\varepsilon }x_{i}^{\varepsilon
}$ we have 
\begin{eqnarray*}
\langle (-y^{\ast },1),(x,g(x))\rangle +\varepsilon \geq
\sum_{i=0}^{n}\lambda _{i}^{\varepsilon }\langle (-y^{\ast
},1),(x_{i}^{\varepsilon },g\left( x_{i}^{\varepsilon }\right) )\rangle  &&
\\
\text{or \qquad }\sum_{i=0}^{n}\lambda _{i}^{\varepsilon }\langle (-y^{\ast
},1),(x,g(x))-(x_{i}^{\varepsilon },g\left( x_{i}^{\varepsilon }\right)
)\rangle  &\geq &-\varepsilon .
\end{eqnarray*}%
As $(-y^{\ast },1$ $)$ supporting $\text{epi }\,h$ we have $\langle (-y^{\ast
},1),(x,g(x))-(x_{i}^{\varepsilon },g\left( x_{i}^{\varepsilon }\right)
)\rangle \leq 0$ for all $i$. Thus for all $i$ with $\lambda
_{i}^{\varepsilon }>0$ we have 
\begin{eqnarray*}
&&\lambda _{i}^{\varepsilon }\langle (-y^{\ast
},1),(x,g(x))-(x_{i}^{\varepsilon },g\left( x_{i}\right) )\rangle  \\
&\geq &\sum_{i=0}^{n}\lambda _{i}^{\varepsilon }\langle (-y^{\ast
},1),(x,g(x))-(x_{i}^{\varepsilon },g\left( x_{i}\right) )\rangle \geq
-\varepsilon .
\end{eqnarray*}%
We have $\langle (-y^{\ast },1),(x,g(x))-(x_{i}^{\varepsilon },g\left(
x_{i}^{\varepsilon }\right) )\rangle \geq -\varepsilon /\lambda
_{i}^{\varepsilon }$ for all $\varepsilon >0$ and $i$. Taking a convergent
subsequence for $\varepsilon \downarrow 0$ we may consider $\lambda
_{i}^{\varepsilon }\rightarrow \lambda _{i}>0$ with $x_{i}^{\varepsilon
}\rightarrow x$ and note that $\langle (-y^{\ast
},1),(x,g(x))-(x_{i},g\left( x_{i}\right) )\rangle \geq 0$ implies $x_{i}=x$
as $\left( x,g\left( x\right) \right) $ is strongly exposed in $\text{epi }g$%
. As we only require to consider convex combination of length $n+1$
(Caratheodory's theorem) we have (using $x_{i}^{\varepsilon }\in D$ a
bounded set and $\lambda _{i}^{\varepsilon }\rightarrow 0$ if $%
x_{i}^{\varepsilon }\not\rightarrow x$) we obtain 
\begin{equation*}
y^{\varepsilon }=\sum_{i=0}^{n}\lambda _{i}^{\varepsilon }x_{i}^{\varepsilon
}\rightarrow \sum_{i=0}^{n}\lambda _{i}x=x\left( \sum_{i=0}^{n}\lambda
_{i}\right) =x
\end{equation*}%
contradicting $\left\Vert (y^{\varepsilon },h(y^{\varepsilon }))-\left(
x,h\left( x\right) \right) \right\Vert \geq \delta >0$ for all $\varepsilon
>0.$

Hence the strongly exposed points of $h$ are dense in $\text{epi }h$. Now the
extremal points of $h$ are contained in the convex closure of the strongly
exposed points i.e. $\overline{\text{co }}\left( \text{ext }\text{epi }h\right)
=\overline{\text{co }}\left( \exp \text{epi }g\right) $. Using Lemma \ref%
{lem:5} the extremal points are actually limits of strongly exposed points.
As strongly exposed points are also extremal points we find that the
extremal points of $h$ are dense in $\text{epi }h.$ For
all 
$
\left( x,h(x)\right) \in \text{ext }\,\text{epi }h=\text{ext }\text{co }\left(
g\right)
$ we have $h\left( x\right) =g\left( x\right) $. Thus $h$ and $g$ coincide on the
dense set of exposed points of $\text{epi }h$ (and $\text{epi }g$). Now apply
Proposition \ref{prop:noconv} to deduce that $\text{epi }h|_{D}:=\overline{%
\exp \text{epi }g|_{D}}.$ That is for $y\in D$ we have 
\begin{equation*}
h\left( y\right) =\liminf_{\substack{ \left( x,h\left( x\right) \right) \in
\exp \text{epi }h  \\ x\in D\text{, }x\rightarrow y}}g\left( x\right) \geq
g\left( y\right) .
\end{equation*}%
But by construction $g\left( y\right) \geq h\left( y\right) $ so $g\left(
y\right) =h\left( y\right) =\text{co }g\left( y\right) $ for $y\in D$. Apply
Lemma \ref{lem:4} for the last observation.
\end{proof}
Let $\bar{\eta}>0$ be the threshold value for which Proposition \ref%
{convexific} holds. Take $\eta >\bar{\eta}$. We wish to the Minty
parametrization of a maximal monotone operator. Consider $y^{\ast }\left(
x\right) =\langle x,z\rangle $ and the problem of minimizing $g\left(
y\right) -\langle y,z\rangle $. The optimality conditions are 
\begin{eqnarray}
0 &\in &\partial f\left( y\right) +\eta \left( y-x^{0}\right) -z  \notag \\
&=&\left[ \partial f+\eta I\right] \left( y\right) -\left( \eta
x^{0}+z\right)  \notag \\
\text{or\qquad }y &\in &\left[ \partial f+\eta I\right] ^{-1}\left( \eta %
\left[ x^{0}+\frac{1}{\eta }z\right] \right) .  \label{neqn:123}
\end{eqnarray}%
As $z$ is arbitrary we may choose $x^{0}+\frac{1}{\eta }z_{x}=x\in \text{lev}%
_{x^{0}}f$ or $z_{x}=\eta \left( x-x^{0}\right) $. The optimality conditions
then become 
\begin{equation*}
y\in \left[ \partial f+\eta I\right] ^{-1}\left( \eta x\right) .
\end{equation*}%
On the other hand consider the problem $y\in P_{\lambda }\left( x\right) \ $%
which has optimality conditions 
\begin{eqnarray}
0 &\in &\partial f\left( y\right) +\frac{1}{\lambda }\left( y-x\right) 
\notag \\
\text{or\qquad }y &\in &\left[ \partial f+\frac{1}{\lambda }I\right]
^{-1}\left( \frac{1}{\lambda }x\right) .  \label{neqn:567}
\end{eqnarray}%
We may then see the correspondence under the identification $\eta =\frac{1}{%
\lambda }$. We may rewrite (\ref{neqn:567}) as 
\begin{eqnarray}
y &\in &\left[ \partial g+\left( \eta -\bar{\eta}\right) I\right]
^{-1}\left( \eta \left( x-x^{0}\right) \right)  \label{neqn:234} \\
&=&\left[ \partial g+\left( \eta -\bar{\eta}\right) I\right] ^{-1}\left(
z_{x}\right)  \notag
\end{eqnarray}

By \eqref{lem:bound} there we may take $\eta $ larger if needed to
ensure for all $x\in B_{\delta _{\eta }}\left( \bar{x}\right) $ and any $%
\bar{x}$ $\in \text{lev}_{x^{0}}f$ we have $y\in P_{\frac{1}{\eta }}\left(
x\right) $ contained in $B_{\varepsilon _{\eta }}\left( \bar{x}\right)
\supseteq B_{\delta _{\eta }}\left( \bar{x}\right) $ where $g\left( x\right)
:=f\left( x\right) +\frac{\eta }{2}\left\Vert x-x^{0}\right\Vert ^{2}$ is
convex on $B_{\delta }\left( \bar{x}\right) \supseteq B_{\varepsilon _{\eta
}}\left( x_{i}\right) .$ Assuming $\text{lev}_{x^{0}}f$ is bounded we may
extract a finite sub-cover of the open cover 
\begin{equation*}
\left\{ B_{\delta _{\eta }}\left( \bar{x}\right) \mid \bar{x}\in \text{lev}%
_{x^{0}}f\right\} .
\end{equation*}%
Let 
\begin{equation*}
\left\{ B_{\delta _{i}}\left( x_{i}\right) \mid i=1,\dots ,k\right\} 
\end{equation*}%
be this finite cover. Then for any $x\in \text{lev}_{x^{0}}f$ we have $x\in
B_{\delta _{i}}\left( x_{i}\right) $ and hence $y\in P_{\frac{1}{\eta }%
}\left( x\right) $ is contained in some $B_{\varepsilon _{i}}\left(
x_{i}\right) \supseteq B_{\delta _{i}}\left( x_{i}\right) $ on which $g$ is
convex. Then (\ref{neqn:234}) applies so via Minty's resolvant theorem 
\begin{equation}
y\in T\left( \eta \left( x-x^{0}\right) \right) :=\left[ \partial g+\left(
\eta -\bar{\eta}\right) I\right] ^{-1}\left( \eta \left( x-x^{0}\right)
\right)   \label{neqn:345}
\end{equation}%
where $T\left( x\right) :=\left[ \partial g+\left( \eta -\bar{\eta}\right) I%
\right] ^{-1}\left( \eta \left( x-x^{0}\right) \right) $ is a single valued
maximal monotone and nonexpansive when $\partial g$ is maximal monotone, see 
\cite[Theorem 12.15]{Rockafellar1998}.
\begin{theorem}\label{conv:th1}
Suppose $f$ is prox-regular and locally Lipschitz on a bounded level set $%
\text{lev}_{ x^0}f$ with $\text{int}$ $\text{lev}_{ x^0}f\neq \emptyset $.
Let $g\left( y;x,a\right)$ be defined in \eqref{g} with $a \geq 0$. There exists an $a^{th}$ and a globally convex function $H(y;x,a)$ satisfying $g(y;x,a)\geq H(y;x,a)$ for all $y\in\mathbbm{R}^n$, $a\geq a^{th}$ and $x\in \text{lev}_{x^0}f$, such that $%
g(y;x,a)$ is the restriction to $\text{lev}_{ x^0}f$ of $H(y;x,a)$.
\end{theorem}
\begin{proof}
We wish to apply Corollary \ref{cor:convex}. We take $D=\text{lev}_{x^{0}}f$
and $g:D\rightarrow \mathbf{R}$ is Lipschitz and prox-regular on the bounded
domain $D$. We need to show that $0\in \partial \left[ g-\langle z,\cdot
\rangle \right] \left( y\right) $ has a unique solution for all $y^{\ast
}\in B_{1}\left( 0\right) $ such that $\left( -z,1\right) $ supports $\text{%
epi }g$ at some $x\in D.$ To this end take $x\in D$ and let $z_{x}=\eta
\left( x-x^{0}\right) $ so that $g\left( y\right) -\langle y,z\rangle $
attains its minimum at $y$ when (\ref{neqn:345}) holds. But by construction
we have $x\in B_{\delta _{i}}\left( x_{i}\right) $ and hence $y\in P_{\frac{1%
}{\eta }}\left( x\right) $ is contained in some $B_{\varepsilon _{i}}\left(
x_{i}\right) \supseteq B_{\delta _{i}}\left( x_{i}\right) $ on which $g$ is
convex. Hence the operator $\ T\left( x\right) :=\left[ \partial g+\left(
\eta -\bar{\eta}\right) I\right] ^{-1}\left( \eta \left( x-x^{0}\right)
\right) $ will have $B_{\delta _{i}}\left( x_{i}\right) $ in its domain and $%
B_{\varepsilon _{i}}\left( x_{i}\right) $ in its range, all contained in a
region on which $\partial g$ is locally a maximal monotone operator (as $g$
is locally convex). Thus $y=T\left( x\right) $ is unique by \cite[Theorem
12.15]{Rockafellar1998}. Thus $\left( x,g\left( x\right) \right) $ is exposed by $%
\left( -z_{x},1\right) $ and Corollary \ref{cor:convex} applies.
\end{proof}
\end{document}